\documentclass[aop]{imsart}

\RequirePackage{amsthm,amsmath,amsfonts,amssymb}
\RequirePackage{bbm}
\RequirePackage[numbers]{natbib}
\RequirePackage{hyperref}
\RequirePackage{pgfplots}
\pgfplotsset{compat=1.17}

\startlocaldefs
\theoremstyle{plain}
\newtheorem{theo}{Theorem}[section]
\newtheorem{lem}{Lemma}[section]

\theoremstyle{remark}
\newtheorem{defi}{Definition}[section]
\newtheorem{remark}{Remark}[section]
\def \b{\beta}
\def \om{\omega}
\def \omij{\omega_{i,j}}

\def \R{\mathbb{R}}

\def \P{\mathbb{P}}

\def \Z{\mathbb{Z}}
\def \E{\mathbb{E}}
\def \Geo{\text{Geo}}
\def \Exp{\textup{Exp}}
\def \Zc{\mathcal{Z}}
\def \Wc{\mathcal{W}}
\def \Bc{\mathcal{B}}
\def \Gc{\mathcal{G}}

\endlocaldefs

\begin{document}

\begin{frontmatter}
%%%%%%%%%%%%%%%%%%%%%%%%%%%%%%%%%%%%%%%%%%%%%%
%%                                          %%
%% Enter the title of your article here     %%
%%                                          %%
%%%%%%%%%%%%%%%%%%%%%%%%%%%%%%%%%%%%%%%%%%%%%%
\title{Universality of directed polymers in the intermediate disorder regime}
%\title{A sample article title with some additional note\thanksref{T1}}
%\runtitle{???}
%\thankstext{T1}{A sample of additional note to the title.}

\begin{aug}
%%%%%%%%%%%%%%%%%%%%%%%%%%%%%%%%%%%%%%%%%%%%%%%
%% Only one address is permitted per author. %%
%% Only division, organization and e-mail is %%
%% included in the address.                  %%
%% Additional information can be included in %%
%% the Acknowledgments section if necessary. %%
%% ORCID can be inserted by command:         %%
%% \orcid{0000-0000-0000-0000}               %%
%%%%%%%%%%%%%%%%%%%%%%%%%%%%%%%%%%%%%%%%%%%%%%%
\author[A]{\fnms{Julian}~\snm{Ransford}\ead[label=e1]{julian.ransford@mail.utoronto.ca}}
%%%%%%%%%%%%%%%%%%%%%%%%%%%%%%%%%%%%%%%%%%%%%%
%% Addresses                                %%
%%%%%%%%%%%%%%%%%%%%%%%%%%%%%%%%%%%%%%%%%%%%%%
\address[A]{Department of Mathematics, University of Toronto\printead[presep={,\ }]{e1}}

\end{aug}

\begin{abstract}
In \cite{krish-qua18}, Krishnan and Quastel showed that the fluctuations of Sepp\"al\"ainen's log-gamma polymer converge in law to the Tracy--Widom GUE distribution in the intermediate disorder regime. This regime corresponds to taking the inverse temperature $\b$ to depend on the length of the polymer $2n$, with $\b=n^{-\alpha}$ for some $\alpha<1/4$. They also conjectured that this should hold for directed polymers with general i.i.d weights. We prove that their conjecture is true for any $1/5<\alpha<1/4$.
\end{abstract}



\end{frontmatter}
%%%%%%%%%%%%%%%%%%%%%%%%%%%%%%%%%%%%%%%%%%%%%%
%% Please use \tableofcontents for articles %%
%% with 50 pages and more                   %%
%%%%%%%%%%%%%%%%%%%%%%%%%%%%%%%%%%%%%%%%%%%%%%
%\tableofcontents

%%%%%%%%%%%%%%%%%%%%%%%%%%%%%%%%%%%%%%%%%%%%%%
%%%% Main text entry area:
\section{Introduction}

The directed polymer was introduced in the physics literature by Huse and Henley during the 80's in \cite{huse-hen85} to model interactions in long molecule chains with impurities. Since then, its study has been taken up by both physicists and mathematicians and it has become one of the central models conjectured to be in the KPZ universality class. See \cite{com-shiga-yosh04} for a survey on the topic.

The model can be described as follows. Let $\xi_{i,j}$, with $i,j \in \Z_{\ge 0}$ be a collection of independent random variables, and let $\b>0$ be a parameter, which is commonly referred to as the \textit{inverse temperature}. We define the \textit{partition function} by
$$Z_{n,\b}=\sum_{\pi: (0,0) \to (n,n)} \prod_{i=0}^{2n} e^{\b \xi_{\pi(i)}}.$$
Here the sum is taken over all up-right paths $\pi$ which start at $(0,0)$ and end at $(n,n)$, that is the set of functions
$$\pi: \{0,1, \dots, 2n\} \to \Z_{\ge 0}^2$$
such that $\pi(0)=(0,0)$, $\pi(2n)=(n,n)$ and for each $i$, $\pi(i+1)-\pi(i)=(1,0)$ or $(0,1)$. See Figure \ref{up-right and Bernoulli paths}. The \textit{free energy} is $\log Z_{n,\b}$, and the \textit{polymer measure} is the random probability measure $\P_{n,\b}$ on up-right paths defined by
$$\P_{n,\b}(\pi)=\frac{\prod_{i=0}^{2n} e^{\b \xi_{\pi(i)}}}{Z_{n,\b}}.$$
The central problem with the directed polymer is to understand the behaviour of the free energy and the polymer measure as a function of $\b$ and as $n \to \infty$.

Unlike simple random walk, the polymer measure exhibits ``superdiffusive'' behaviour, which loosely means that the transversal fluctuations of a $\P_{n,\b}$ distributed path are of order $n^{\zeta}$ for some large fixed exponent $\zeta$. When $\b=0$, the polymer measure is random walk bridge measure, and $\zeta=1/2$. For $\b>0$, the conjecture is that $\zeta=2/3$. This was verified numerically in Huse and Henley's original paper \cite{huse-hen85}, but it has only been proven for the very special log-gamma polymer of Sepp\"al\"ainen \cite{bar-cor-dim21,bor-cor-rem13,seppa12}, and for general weights in thin rectangles \cite{auf-bai-cor12}.

As for the free energy, it too is expected to have fluctuations of order $n^{\chi}$ for some fixed exponent $\chi$. The conjecture here is that $\chi=1/3$. Moreover, the limiting distribution of the scaled fluctuations is conjectured to be Tracy--Widom GUE. This has only been confirmed for the log-gamma polymer again \cite{bar-cor-dim21,bor-cor-rem13,seppa12} as well as for certain models of last passage percolation (i.e the limit $\log Z_{n,\b}/\b$ as $\b \to \infty$) \cite{BDJ99,johan00,johan00-2}. All of these and other models with the same fluctuation exponents ($\zeta=2/3,\chi=1/3$) are conjectured to be in the Kardar--Parisi--Zhang universality class. See \cite{cor12} for a survey on KPZ theory.

In this paper, we consider the \textit{intermediate disorder regime} of the directed polymer. This consists in taking $\b=\b_n$ depending on $n$ and such that $\b_n \to 0$ as $n \to \infty$. This regime was introduced by Alberts, Khanin and Quastel in \cite{AKQ14}, and they showed in \cite{AKQ14,AKQ14-2} that for $\b_n \sim \b n^{-1/4}$, the directed polymer converges to a non-trivial limit known as the \textit{continuum directed polymer}. 

If $\b_n=n^{-\alpha}$ for some $\alpha>1/4$, then the limiting fluctuations of the free energy are Gaussian. This can be seen by doing a Taylor expansion of $\log Z_{n,\b}$ centred at $\b=0$:
$$\log Z_{n,\b}=\log \binom{2n}{n}+\left(\sum_{i+j \le n} \P((i,j) \in S) \xi_{ij}\right)\b+ \dots,$$
for $S$ uniformly distributed on up-right paths from $(0,0)$ to $(n,n)$. One can check that the variance of the coefficient of the $k$-th order term is of order $n^{k/2}$. So with $\alpha>1/4$, the terms of order 2 and higher vanish in the limit, and the order 1 term is a weighted sum of independent random variables, so it has Gaussian fluctuations by the central limit theorem. For $\alpha=1/4$, the terms of higher order do not vanish; in fact, each term in the expansion converges in distribution individually, and Alberts, Khanin and Quastel showed that the entire sum converges.

When $\alpha<1/4$, it was conjectured that the directed polymer should fall back in the KPZ regime and that the fluctuations are of order $\b^{4/3}n^{1/3}$ with a Tracy--Widom GUE limit. This was proven by Krishnan and Quastel in \cite{krish-qua18} for the log-gamma polymer as well as for polymers whose weights match at least 9 moments with those of the log-gamma. In this paper, we show that this last requirement is unnecessary, and that the fluctuations are indeed of this order with Tracy--Widom GUE limit for $1/5<\alpha<1/4$ and for any choice of weights with exponential moments.

\section{Main results}

We consider collections of independent random variables $\omij(\b)$, $i,j \ge 0$, which are parametrized by the inverse temperature $\b$. The usual directed polymer corresponds to taking $\omij(\b)=e^{\b \xi_{i,j}}$ where the $\xi_{i,j}$ are independent random variables. We allow this more complicated dependence on $\b$ in order to include the log-gamma polymer, which is not given by a standard polymer. A \textit{path} is a function $\pi: \Z_{\ge 0} \to \Z$ such that $\pi(0)=0$ and $\pi(i)-\pi(i-1)\in \{0,1\}$ for all $i$. For an inverse temperature $\b>0$, we define the \textit{partition function} as
$$Z_{n,\b}=\sum_{\pi} \prod_{i=0}^{2n} \om_{i,\pi(i)}(\b).$$
Here the sum is taken over all paths of length $2n$ such that $\pi(2n)=n$. For convenience, we translated the polymer in terms of Bernoulli paths instead of up-right paths. There is an obvious bijection between the set of up-right paths from $(0,0)$ to $(n,n)$ and the set of Bernoulli paths of length $2n$ started at 0 and ending at $n$, so this transformation does not change anything. This translation however will make a lot of the expressions easier to write down in terms of intersections of Bernoulli walks.

\begin{figure}
\centering
	\begin{tikzpicture}
	\draw[very thin,color=gray] (0,0) grid (4,4);
	\draw[->] (0,0) -- (4.3,0) node[right] {};
	\draw[->] (0,0) -- (0,4.3) node[above] {};
	
	\foreach \i in {0,1,...,4}{
	\coordinate[label=below:$\i$] (x) at (\i,0);
	}
	
	\foreach \i in {0,1,...,4}{
	\coordinate[label=left:$\i$] (x) at (0,\i);
	}
	
	\filldraw[black]  (4,4) circle (3pt);
	
	\draw[-,color=black,line width=0.75mm] (0,0) -- (1,0) node[right]{};
	\draw[-,color=black,line width=0.75mm] (1,0) -- (1,2) node[right]{};
	\draw[-,color=black,line width=0.75mm] (1,2) -- (3,2) node[right]{};
	\draw[-,color=black,line width=0.75mm] (3,2) -- (3,3) node[right]{};
	\draw[-,color=black,line width=0.75mm] (3,3) -- (4,3) node[right]{};
	\draw[-,color=black,line width=0.75mm] (4,3) -- (4,4) node[right]{};	
	
	\end{tikzpicture}
	\begin{tikzpicture}
	\draw[very thin,color=gray] (0,0) grid (8,4);
	\draw[->] (0,0) -- (8.3,0) node[right] {};
	\draw[->] (0,0) -- (0,4.3) node[above] {};
	
	\foreach \i in {0,1,...,8}{
	\coordinate[label=below:$\i$] (x) at (\i,0);
	}
	
	\foreach \i in {0,1,...,4}{
	\coordinate[label=left:$\i$] (x) at (0,\i);
	}
	
	\filldraw[black]  (8,4) circle (3pt);
	
	\draw[-,color=black,line width=0.75mm] (0,0) -- (1,0) node[right]{};
	\draw[-,color=black,line width=0.75mm] (1,0) -- (3,2) node[right]{};
	\draw[-,color=black,line width=0.75mm] (3,2) -- (5,2) node[right]{};
	\draw[-,color=black,line width=0.75mm] (5,2) -- (6,3) node[right]{};
	\draw[-,color=black,line width=0.75mm] (6,3) -- (7,3) node[right]{};
	\draw[-,color=black,line width=0.75mm] (7,3) -- (8,4) node[right]{};	
	
	\end{tikzpicture}

\caption{On the left: an up-right path from (0,0) to (4,4), and on the right the corresponding Bernoulli path from (0,0) to (8,4)}
\label{up-right and Bernoulli paths}
\end{figure}
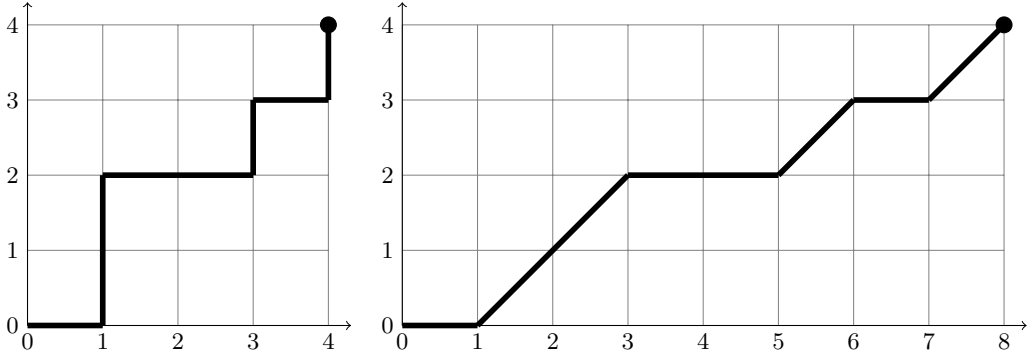

If the weights are all positive, we can define the corresponding \textit{polymer measure} $\P_{n,\b}$ on paths as follows:
$$\P_{n,\b}(\pi)=\frac{\prod_{i=0}^{2n} \om_{i,\pi(i)}(\b)}{Z_{n,\b}}.$$
We are primarily interested in the \textit{intermediate disorder regime}, which corresponds to taking $\b=\b_n \to 0$ as $n \to \infty$. In view of the discussion in the introduction, we will consider the case when $\b=n^{-\alpha}$ for some fixed $\alpha<1/4$. Our main result is the following.

\begin{theo}\label{TW thm}
Let $\alpha \in (1/5,1/4)$, and set $\b=n^{-\alpha}$. Let $\xi_{i,j}, i,j \ge 0$ be i.i.d random variables with variance $\sigma^2$ and an exponential moment, that is $\E(e^{c\xi_{1,1}})<\infty$ for some $c>0$, and let $Z_{n,\b}$ be the partition function for the corresponding directed polymer. Then
$$\frac{\log Z_{n,\b}-a_n}{(4\sigma^4\b^4n)^{1/3}} \xrightarrow{d} \text{TW}_{GUE}$$
where
$$a_n=2n\left(\log\psi(\b)+\log 2-\frac{\sigma^4 \b^4}{3}\right)$$
and $\psi$ is the moment generating function of the $\xi_{i,j}$'s.
\end{theo}

The asymptotic free energy $a_n$ depends on the distribution of the weights, but only via the moment generating function. By expanding $\log \psi(\b)$ as a Taylor series, one can see that the $\b^5$ terms and higher can be ignored, since in the regime $1/5<\alpha<1/4$, we have that $n\b^5$ will be of smaller order then $\b^{4/3}n^{1/3}$. The Taylor coefficients of $\log \psi(\b)$ are given by the cumulants, and those depend only on the moments of the same order and lower, and so $a_n$ only depends on the first 4 moments of the $\xi_{i,j}$'s.

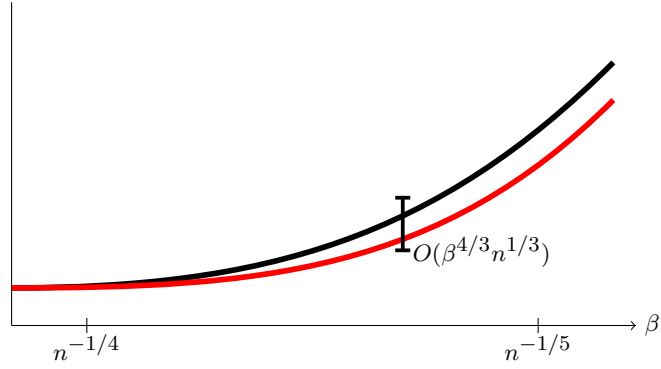
\begin{figure}
\centering
	\begin{tikzpicture}
	%\draw[very thin,color=gray] (0,0) grid (8,4);
	\draw[->] (0,0) -- (8.3,0) node[right] {$\beta$};
	\draw[-] (0,0) -- (0,4.3) node[above] {};
	
	\coordinate[label=below:$n^{-1/4}$] (x) at (1,0);
	\draw[-] (1,-0.1) -- (1,0.1);
	\coordinate[label=below:$n^{-1/5}$] (x) at (7,0);
	\draw[-] (7,-0.1) -- (7,0.1);
	
	%\foreach \i in {0,1,...,4}{
	%\coordinate[label=left:$\i$] (x) at (0,\i);
	%}
	
	%Limit shape
	\draw[-,color=black,line width=0.75mm] (0,0.5) to[out=0, in=225] (8,3.5) node[right]{};
	
	%Free energy
	\draw[-,color=red,line width=0.75mm] (0,0.5) to[out=0, in=225] (8,3) node[right]{};
	
	%Size of fluctuations
	\draw[-,color=black,line width=0.5mm] (5.2,1) -- (5.2,1.7) node[right, yshift=-20]{$O(\beta^{4/3}n^{1/3})$};
	\draw[-,color=black,line width=0.5mm] (5.1,1) -- (5.3,1);
	\draw[-,color=black,line width=0.5mm] (5.1,1.7) -- (5.3,1.7);

	\end{tikzpicture}

\caption{Fluctuations in the valid $\beta$ window. The black curve represents the theoretical asymptotic free energy (the $a_n$'s) and the red curve the free energy.}
\label{limit shape figure}
\end{figure}

The significance of this result is illustrated in Figure \ref{limit shape figure}. The fluctuations of the directed polymer are asymptotically Tracy--Widom GUE for an inverse temperature $\b$ in the range $(n^{-1/4+\epsilon},n^{-1/5-\epsilon})$, for every $\epsilon>0$. The key is that Theorem \ref{TW thm} is \textit{universal}, and does not depend on the choice of distributions. As discussed in the introduction, there is very little that is known about the directed polymer with general weights.

\begin{defi}
A collection of independent, parametrized random variables $\omij(\b)$, $i,j \ge 0$ is called \textit{valid} if it satisfies the following conditions for all $i,j \ge 0$ and for all sufficiently small $\b>0$:
\begin{enumerate}
\item $\omij(\b)>0$ almost surely.
\item $\E(\omij(\b))=1$.
\item There are constants $C_k$ such that
$$\E(|\omij(\b)-1|^k|)\le C_k \b^k$$
\item There are positive constants $c_1,c_2,c_3$ such that for all $0<s<1$ there is an $A(s)>0$ with
$$\P(e^{-c_1 \b^s} \le \omij(\b) \le e^{c_2 \b^s}) \ge 1-A(s) e^{-c_3\b^{s-1}}$$
for all sufficiently small $\b>0$.
\end{enumerate}
\end{defi}

\begin{remark}
It is sufficient for weights to be only defined along some countable sequence of parameters $\b$ converging to 0. This way all the different sets of weights for all parameters can be coupled on the same probability space.
\end{remark}
\begin{remark}
The specific form of Condition 4 is not extremely important. What matters is that $\omij(\b)$ should be approximately $e^{-c\b}$ for $\b$ sufficiently small, and the probability that this is not the case should go to 0 with $n$ faster than any polynomial (recall that we are considering $\b$ of the form $\b=n^{-\alpha}$).
\end{remark}
\begin{remark}\label{combination of valid weights}
If $\omij(\b)$ and $\omij'(\b)$ are two independent collections of valid weights, then any combination of the two sets of weights is also valid. More precisely, let $G:\Z^2 \to \{0,1\}$ be an arbitrary function, and define 
$$\omij''(\b)=
\begin{cases}
\omij(\b) & \quad \text{if } G(i,j)=0\\
\omij'(\b) & \quad \text{if } G(i,j)=1.
\end{cases}
$$ 
Then $\omij''(\b)$ is valid. In fact, it is easy to see that one can pick the constants in Conditions 3 and 4 so that those inequalities work uniformly over all choices of $G$.
\end{remark}

The main ingredient that goes in the proof of Theorem \ref{TW thm} is the following perturbation theorem.
\begin{theo}\label{perturbation thm}
Let $\alpha \in (1/5,1/4)$, and set $\b=n^{-\alpha}$. Let $\omij(\b)$ and $\omij'(\b)$ be independent valid collections of weights such that for each $i,j \ge 0$, and for every sufficiently small $\b$, $\E(\omij(\b)^2)=\E(\omij'(\b)^2)$. Then for a probability distribution $F$ on $\R$ and a sequence of real numbers $(a_n)$, we have
$$\frac{\log Z_{n,\b}-a_n}{\b^{4/3}n^{1/3}} \xrightarrow{d} F$$
if and only if
$$\frac{\log Z_{n,\b}'-a_n}{\b^{4/3}n^{1/3}} \xrightarrow{d} F.$$
\end{theo}

In Section \ref{Tracy-Widom fluctuations section}, we will show that both the standard directed polymer and the log-gamma polymer (or at least suitably normalized versions of them) are valid. Since the fluctuations of the log-gamma polymer are known to converge to the Tracy--Widom GUE distribution, this will imply Theorem \ref{TW thm}. Another natural choice of weights is $\omij(\b)=1+\b\xi_{i,j}$ where the $\xi_{i,j}$'s are i.i.d with mean 0. This also turns out to be a valid set of weights, and Theorem \ref{TW thm} will hold as well for its free energy, but with $\log \psi(\b)$ removed from the $a_n$'s: $a_n=2n(\log2-\sigma^4 \b^4/3)$. The calculations are very similar to those for the standard directed polymer done in Section \ref{Tracy-Widom fluctuations section}, so we omit them.

The reason for which the lower exponent $\alpha>1/5$ comes up is due to a somewhat crude tail estimate for the partition function $Z_{n,\b}$. As the calculations at the end of Section \ref{perturbation proof section} show, Theorems \ref{TW thm} and \ref{perturbation thm} should hold with $2/17<\alpha<1/4$, and we conjecture that this is indeed the case. If we strengthen the assumptions of Theorem \ref{perturbation thm} to $\omij(\b)$ and $\omij'(\b)$ having the same first $k$ moments, then Theorems \ref{TW thm} and \ref{perturbation thm} should hold (with different $a_n$'s) for $\alpha>2/(3k+11)$; see Remark \ref{alpha for higher moments} at the end of Section \ref{perturbation proof section}.

\section{Ideas behind the proof}

The main tool in the proof of Theorem \ref{perturbation thm} is the Lindeberg method. Suppose that $g_n: \R^n \to \R$ is a sequence of functions and $X_1, X_2, \dots$ is a sequence of independent random variables for which we know $g_n(X_1, \dots, X_n) \xrightarrow{d} F$. The Lindeberg method consists in showing that $g_n(Y_1, \dots, Y_n)$ also converges in distribution to $F$ for a different sequence of independent random variables by estimating the error when one changes just one of the inputs of $g_n$ from the $X$ sequence to the $Y$ sequence. If the sum of all the errors after changing each input one by one is $o(1)$ in some metric that metrizes convergence in distribution, then we can conclude that $g_n(Y_1, \dots, Y_n) \xrightarrow{d} F$. This is typically done by expanding $\E(f \circ g_n(X_1, \dots, X_n))$ (here $f$ is some test function) as a Taylor series, and then using some sort of moment matching condition to cancel certain terms. In our case, $g_n$ will correspond to $\log Z_{n,\b}$.

Consider the polymer with weights $\omij(\b)$. For fixed $i$ and $j$, we introduce the following notation
\begin{eqnarray*}
W_n(i,j)&=&\sum_{\pi, (i,j) \in \pi} \prod_{\ell \neq i} \om_{\ell, \pi(\ell)}(\b)\\
V_n(i,j)&=&\sum_{\pi, (i,j) \notin \pi} \prod_{\ell=1}^n \om_{\ell, \pi(\ell)}(\b).
\end{eqnarray*}
The first sum is taken over the paths which go through the point $(i,j)$, but the term in the product corresponding to $\omij(\b)$ is removed. The sum for $V_n(i,j)$ is taken over the paths that do not go through $(i,j)$. We clearly have $Z_{n,\b}=V_n(i,j)+\omij(\b) W_n(i,j)$.

By virtue of \cite[Theorem 2.1]{bill99} and a standard density argument, it is enough to check that for all $C^3$ functions $f$ with bounded derivatives,
$$\lim_{n \to \infty} \E(f(X_n))=\lim_{n \to \infty} \E(f(Y_n))$$
in order to show that two sequences $(X_n)$ and $(Y_n)$ converge in distribution to the same thing. So let $f: \R \to \R$ be such a function. Then, by Taylor's theorem (we omit the dependence on $i,j$ and $\b$ to make the notation easier to follow),

\begin{eqnarray*}
&&f\left(\frac{\log Z_n-a_n}{\sigma}\right)=f\left(\frac{\log(V_n+W_n)-a_n}{\sigma}\right)\\
&&+f'\left(\frac{\log(V_n+W_n)-a_n}{\sigma}\right)\frac{W_n}{\sigma(V_n+W_n)} (\omij-1)\\
&&-\frac{1}{2}\left[f'\left(\frac{\log(V_n+W_n)-a_n}{\sigma}\right)\right.\\
&&\phantom{-}\left.-\sigma^{-1}f''\left(\frac{\log(V_n+W_n)-a_n}{\sigma}\right)\right] \frac{W_n^2}{\sigma(V_n+W_n)^2} (\omij-1)^2\\
&&+\frac{1}{6}\left[2f'\left(\frac{\log(V_n+\eta W_n)-a_n}{\sigma}\right)\right.\\
&&\phantom{-}-3\sigma^{-1} f''\left(\frac{\log(V_n+\eta W_n)-a_n}{\sigma}\right)\\
&&\phantom{+}\left.+\sigma^{-2} f'''\left(\frac{\log(V_n+\eta W_n)-a_n}{\sigma}\right)\right] \frac{W_n^3}{\sigma(V_n+\eta W_n)^3}(\omij-1)^3
\end{eqnarray*}

for some $\eta$ between 1 and $\omij(\b)$. Here $\sigma=\b^{4/3} n^{1/3}$ is the scaling factor that appears in Theorem \ref{perturbation thm}. We can do exactly the same but with only the weight $\omij$ replaced by $\omij'$ in the polymer. Then, using the fact that the first two moments match, that $W_n$ and $V_n$ are independent of $\omij$ and $\omij'$, that the derivatives of $f$ are bounded, and Condition 3 for valid sets of weights applied to the third moment, we obtain

\begin{equation}\label{taylor bound}
\begin{split}
&\E\left|f\left(\frac{\log(V_n+\omij W_n)-a_n}{\sigma}\right)-f\left(\frac{\log(V_n+\omij' W_n)-a_n}{\sigma}\right)\right|\\
&\le \frac{C\b^3}{\sigma} \E\left( \frac{W_n^3}{(V_n+\eta W_n)^3}+\frac{W_n^3}{(V_n+\eta' W_n)^3} \right),
\end{split}
\end{equation}

for some number $\eta'$ between 1 and $\omij'(\b)$. Here and in what follows, we will use $C$ to denote an unspecified constant that may change from line to line. We can get rid of the dependence on $\eta$ by noting that if $\omij \le 1$, then $\omij \le \eta$, so $V_n+\eta W_n \ge V_n+\omij W_n=Z_n$, and if $\omij>1$, then $\eta \ge 1$, so $V_n+\eta W_n \ge V_n+W_n$, and likewise for $\eta'$. So the term in \eqref{taylor bound} can be bounded above by

\begin{equation}\label{taylor bound 2}
\frac{C \b^3}{\sigma} \E\left( \frac{W_n^3}{Z_n^3}+\frac{W_n^3}{(V_n+W_n)^3} \right).
\end{equation}

The first term in the expectation is the probability that three independent paths picked according to the polymer measure all intersect at the point $(i,j)$, and the second is the same thing except the weight $\omij(\b)$ has been replaced by 1. After adding all the errors over all $i,j$, we get the expected number of times that three independent polymer distributed paths intersect simultaneously. If the polymer measure was the same as Bernoulli random walk, then this would be of order $\log n$, but unfortunately that is not the case and the analysis requires some more work.

To illustrate some of the ideas that go into the later calculations, suppose that we are in the supercritical case $\b=n^{-\alpha}$ for some $\alpha>1/4$, say $\alpha=1/4+\delta$, and let us compute the variance of the normalized partition function $Z_{n,\b}/\binom{2n}{n}$. We clearly have $\E(Z_{n,\b}/\binom{2n}{n})=1$, and so

\begin{eqnarray*}
\text{Var}\left(\frac{Z_{n,\b}}{\binom{2n}{n}}\right)&=&\E\left(\frac{1}{\binom{2n}{n}} \sum_{\pi} \prod_{i=0}^{2n} \om_{i,\pi(i)}(\b) \right)^2-1\\
&=&\E\left(\frac{1}{\binom{2n}{n}^2}\sum_{\pi_1,\pi_2}\prod_{i=0}^{2n} \om_{i,\pi_1(i)}\om_{i,\pi_2(i)}\right)-1\\
&=&\frac{1}{\binom{2n}{n}^2}\sum_{\pi_1,\pi_2}\prod_{i=0}^{2n} \E(\om_{i,\pi_1(i)}\om_{i,\pi_2(i)})-1.
\end{eqnarray*}

The second and third sums are taken over pairs of paths $\pi_1, \pi_2$. If $i \neq j$, then the weights $\om_{i,k}$ and $\om_{j,\ell}$ are necessarily distinct no matter what $k$ and $\ell$ are, hence they are independent which is why the expectation can be switched with the product in this way. We can further simplify the expectation of each term
$$\E(\om_{i,\pi_1(i)}\om_{i,\pi_2(i)})=\E(\om_{i,\pi_1(i)})\E(\om_{i,\pi_2(i)})=1$$
but only if $\pi_1$ and $\pi_2$ do not intersect at time $i$. If they do intersect at time $i$, then we instead get the second moment of $\om_{i,\pi_1(i)}$, which in view of Condition 3 for valid sets of weights, should be bounded above by $1+C\b^2$. We therefore get
$$\text{Var}\left(\frac{Z_{n,\b}}{\binom{2n}{n}}\right) \le \frac{1}{\binom{2n}{n}^2} \sum_{\pi_1, \pi_2} (1+C\b^2)^{L(\pi_1,\pi_2)}-1,$$
where $L(\pi_1,\pi_2)$ is the \textit{intersection local time} of $\pi_1,\pi_2$, i.e the number of times that they intersect. The sum can be interpreted as an expectation over pairs of paths selected using Bernoulli random walk bridge measure, and in that case, it turns out that $L(\pi_1,\pi_2)=O(\sqrt{n})$ with high probability, so that
$$\text{Var}\left(\frac{Z_{n,\b}}{\binom{2n}{n}}\right) \le (1+C\b^2)^{C\sqrt{n}}-1=\left(1+\frac{C}{n^{1/2+2\delta}}\right)^{C\sqrt{n}}-1\le \frac{C}{n^{2\delta}}.$$
More generally, the same calculations can be done for the centred $k$-th moment of $Z_{n,\b}/\binom{2n}{n}$. When expanding the $k$-th power, we instead get a sum over $k$-tuples of paths, weighted by exponentials of the number of times that at least two paths intersect. These exponentials are more complicated since it is possible for 3 or more paths to intersect simultaneously, but these occurrences are much rarer. With some work, one can then obtain
\begin{equation}\label{kth centred moment}
\E \left| \frac{Z_{n,\b}}{\binom{2n}{n}}-1\right|^k \le \frac{C_k}{n^{k\delta}}.
\end{equation}
The constants $C_k$ depend on $k$, but not on $n$. Hence, using the obvious bound $W_n\le Z_n$ and Markov's inequality,
\begin{eqnarray*}
\E\left( \frac{W_n^3}{Z_n^3}\right)&=&\E\left(\frac{(W_n/\binom{2n}{n})^3}{(Z_n/\binom{2n}{n})^3}\right)\\
&=&\E\left(\frac{(W_n/\binom{2n}{n})^3}{(Z_n/\binom{2n}{n})^3}\mathbbm{1}_{\{|Z_n/\binom{2n}{n}-1|< 1/2\}}\right)\\
&\phantom{=}&+\E\left(\frac{(W_n/\binom{2n}{n})^3}{(Z_n/\binom{2n}{n})^3}\mathbbm{1}_{\{|Z_n/\binom{2n}{n}-1|\ge 1/2\}}\right)\\
&\le& 8\E\left(\frac{W_n}{\binom{2n}{n}}\right)^3+\P\left(\left|\frac{Z_n}{\binom{2n}{n}}-1\right| \ge \frac{1}{2}\right)\\
&\le& 8\E\left(\frac{W_n}{\binom{2n}{n}}\right)^3+\frac{C_k}{n^{k\delta}}.
\end{eqnarray*}
The third moment of $W_n/\binom{2n}{n}$ can be estimated in a similar way to how we obtained \eqref{kth centred moment} (by considering paths conditioned to go through the specific point $(i,j)$). For the second term, we simply pick a large enough $k$ so that $n^{-k\delta}$ is of lower order than the third moment of $W_n/\binom{2n}{n}$. A similar calculation can be done for $\E(W_n/(V_n+W_n))^3$, and this gives a bound for \eqref{taylor bound 2}.

Of course, this all only works when $\alpha>1/4$, which is not our case. However this argument can still be used by working \textit{locally} instead. Let $[a,b]$ be some interval contained in $[0,2n]$, and let $n_0=b-a$. Then $Z_{n,\b}$ can be rewritten as
\begin{equation}\label{partition function other form}
Z_{\mu}(\b):=\frac{Z_{n,\b}}{\Zc}=\sum_{\pi:a \to b} \mu(\pi) \prod_{\ell \in [a,b]} \om_{\ell,\pi(\ell)}(\b),
\end{equation}
where
\begin{equation}\label{mu definition}
\mu(\pi)=\frac{1}{\Zc} \sum_{\widetilde{\pi}: \left. \widetilde{\pi}\right|_{[a,b]}=\pi} \,\,\prod_{\ell \notin [a,b]}  \om_{\ell, \widetilde{\pi}(\ell)}(\b), \quad \Zc=\sum_{\pi} \prod_{i \notin [a,b]} \om_{i,\pi(i)}(\b).
\end{equation}
The summation for $Z_{n,\b}/\Zc$ is taken over paths on $[a,b]$, the one for $\Zc$ is over all paths on $[0,2n]$, and the one for $\mu$ is over paths $\widetilde{\pi}$ on $[0,2n]$ whose restriction to $[a,b]$ is $\pi$. If we condition on the weights lying outside the strip $[a,b]\times \Z$, then $\mu$ is a probability measure on paths in $[a,b]$, so that $Z_{\mu}(\b)$ can be interpreted as the partition function for the polymer where the paths are on $[a,b]$ and are weighted according to $\mu$. The idea is that if $n_0$ is small enough, then we'll have $\b \sim n_0^{-\gamma}$ for an exponent $\gamma$ that is larger than $1/4$, and so the calculations above work and we get a good estimate in \eqref{taylor bound 2}.

The only slight wrinkle now is that the paths are distributed according to $\mu$, not Bernoulli walk measure, so obtaining \eqref{kth centred moment} requires some extra work. Concretely, $\mu$ is simply a Bernoulli bridge measure where the endpoints are random, chosen depending on what the environment is outside of the strip $[a,b]$. Indeed, two paths on $[a,b]$ with the same endpoints have the same measure, since the paths on $[0,2n]$ with those restrictions on $[a,b]$ must agree outside of $[a,b]$, and $\mu$ only counts weights from the environment outside the strip. See Figure \ref{mu paths figure}.

 \begin{figure}
\centering
\begin{tikzpicture}
	\draw[->] (0,0) -- (8.3,0) node[right] {};
	\draw[->] (0,0) -- (0,7.3) node[above] {};
	
	\foreach \i\j\k\l in {2.0/1.0/2.04/1.1, 2.04/1.1/2.08/1.1, 2.08/1.1/2.12/1.2, 2.12/1.2/2.16/1.2, 2.16/1.2/2.2/1.2, 2.2/1.2/2.24/1.2, 2.24/1.2/2.28/1.3, 2.28/1.3/2.32/1.4, 2.32/1.4/2.36/1.5, 2.36/1.5/2.4/1.6, 2.4/1.6/2.44/1.6, 2.44/1.6/2.48/1.6, 2.48/1.6/2.52/1.6, 2.52/1.6/2.56/1.7, 2.56/1.7/2.6/1.7, 2.6/1.7/2.64/1.7, 2.64/1.7/2.68/1.8, 2.68/1.8/2.72/1.8, 2.72/1.8/2.76/1.9, 2.76/1.9/2.8/2.0, 2.8/2.0/2.84/2.1, 2.84/2.1/2.88/2.1, 2.88/2.1/2.92/2.1, 2.92/2.1/2.96/2.2, 2.96/2.2/3.0/2.3, 3.0/2.3/3.04/2.4, 3.04/2.4/3.08/2.5, 3.08/2.5/3.12/2.5, 3.12/2.5/3.16/2.5, 3.16/2.5/3.2/2.5, 3.2/2.5/3.24/2.5, 3.24/2.5/3.28/2.5, 3.28/2.5/3.32/2.5, 3.32/2.5/3.36/2.5, 3.36/2.5/3.4/2.5, 3.4/2.5/3.44/2.5, 3.44/2.5/3.48/2.5, 3.48/2.5/3.52/2.5, 3.52/2.5/3.56/2.6, 3.56/2.6/3.6/2.6, 3.6/2.6/3.64/2.6, 3.64/2.6/3.68/2.7, 3.68/2.7/3.72/2.8, 3.72/2.8/3.76/2.9, 3.76/2.9/3.8/3.0, 3.8/3.0/3.84/3.1, 3.84/3.1/3.88/3.2, 3.88/3.2/3.92/3.2, 3.92/3.2/3.96/3.3, 3.96/3.3/4.0/3.3, 4.0/3.3/4.04/3.3, 4.04/3.3/4.08/3.3, 4.08/3.3/4.12/3.4, 4.12/3.4/4.16/3.5, 4.16/3.5/4.2/3.6, 4.2/3.6/4.24/3.7, 4.24/3.7/4.28/3.7, 4.28/3.7/4.32/3.7, 4.32/3.7/4.36/3.7, 4.36/3.7/4.4/3.8, 4.4/3.8/4.44/3.9, 4.44/3.9/4.48/4.0, 4.48/4.0/4.52/4.1, 4.52/4.1/4.56/4.1, 4.56/4.1/4.6/4.2, 4.6/4.2/4.64/4.2, 4.64/4.2/4.68/4.2, 4.68/4.2/4.72/4.3, 4.72/4.3/4.76/4.3, 4.76/4.3/4.8/4.3, 4.8/4.3/4.84/4.4, 4.84/4.4/4.88/4.4, 4.88/4.4/4.92/4.4, 4.92/4.4/4.96/4.4, 4.96/4.4/5.0/4.5, 5.0/4.5/5.04/4.6, 5.04/4.6/5.08/4.7, 5.08/4.7/5.12/4.7, 5.12/4.7/5.16/4.7, 5.16/4.7/5.2/4.8, 5.2/4.8/5.24/4.8, 5.24/4.8/5.28/4.9, 5.28/4.9/5.32/5.0, 5.32/5.0/5.36/5.1, 5.36/5.1/5.4/5.2, 5.4/5.2/5.44/5.3, 5.44/5.3/5.48/5.4, 5.48/5.4/5.52/5.4, 5.52/5.4/5.56/5.4, 5.56/5.4/5.6/5.5, 5.6/5.5/5.64/5.6, 5.64/5.6/5.68/5.7, 5.68/5.7/5.72/5.8, 5.72/5.8/5.76/5.8, 5.76/5.8/5.8/5.8, 5.8/5.8/5.84/5.9, 5.84/5.9/5.88/5.9, 5.88/5.9/5.92/5.9, 5.92/5.9/5.96/6.0}{
	\draw[-, color=blue, line width=0.5mm] (\i,\j) -- (\k,\l);
	}
	
	\foreach \i\j\k\l in {2.0/2.0/2.04/2.0, 2.04/2.0/2.08/2.0, 2.08/2.0/2.12/2.1, 2.12/2.1/2.16/2.1, 2.16/2.1/2.2/2.2, 2.2/2.2/2.24/2.2, 2.24/2.2/2.28/2.3, 2.28/2.3/2.32/2.3, 2.32/2.3/2.36/2.3, 2.36/2.3/2.4/2.3, 2.4/2.3/2.44/2.4, 2.44/2.4/2.48/2.4, 2.48/2.4/2.52/2.4, 2.52/2.4/2.56/2.4, 2.56/2.4/2.6/2.4, 2.6/2.4/2.64/2.5, 2.64/2.5/2.68/2.5, 2.68/2.5/2.72/2.5, 2.72/2.5/2.76/2.6, 2.76/2.6/2.8/2.7, 2.8/2.7/2.84/2.8, 2.84/2.8/2.88/2.8, 2.88/2.8/2.92/2.9, 2.92/2.9/2.96/2.9, 2.96/2.9/3.0/2.9, 3.0/2.9/3.04/2.9, 3.04/2.9/3.08/2.9, 3.08/2.9/3.12/2.9, 3.12/2.9/3.16/2.9, 3.16/2.9/3.2/3.0, 3.2/3.0/3.24/3.1, 3.24/3.1/3.28/3.1, 3.28/3.1/3.32/3.1, 3.32/3.1/3.36/3.1, 3.36/3.1/3.4/3.2, 3.4/3.2/3.44/3.2, 3.44/3.2/3.48/3.2, 3.48/3.2/3.52/3.2, 3.52/3.2/3.56/3.2, 3.56/3.2/3.6/3.2, 3.6/3.2/3.64/3.2, 3.64/3.2/3.68/3.2, 3.68/3.2/3.72/3.3, 3.72/3.3/3.76/3.3, 3.76/3.3/3.8/3.3, 3.8/3.3/3.84/3.3, 3.84/3.3/3.88/3.3, 3.88/3.3/3.92/3.3, 3.92/3.3/3.96/3.3, 3.96/3.3/4.0/3.3, 4.0/3.3/4.04/3.3, 4.04/3.3/4.08/3.4, 4.08/3.4/4.12/3.4, 4.12/3.4/4.16/3.4, 4.16/3.4/4.2/3.4, 4.2/3.4/4.24/3.4, 4.24/3.4/4.28/3.5, 4.28/3.5/4.32/3.5, 4.32/3.5/4.36/3.5, 4.36/3.5/4.4/3.6, 4.4/3.6/4.44/3.7, 4.44/3.7/4.48/3.7, 4.48/3.7/4.52/3.7, 4.52/3.7/4.56/3.7, 4.56/3.7/4.6/3.8, 4.6/3.8/4.64/3.8, 4.64/3.8/4.68/3.8, 4.68/3.8/4.72/3.9, 4.72/3.9/4.76/3.9, 4.76/3.9/4.8/4.0, 4.8/4.0/4.84/4.0, 4.84/4.0/4.88/4.0, 4.88/4.0/4.92/4.0, 4.92/4.0/4.96/4.0, 4.96/4.0/5.0/4.0, 5.0/4.0/5.04/4.1, 5.04/4.1/5.08/4.1, 5.08/4.1/5.12/4.2, 5.12/4.2/5.16/4.3, 5.16/4.3/5.2/4.3, 5.2/4.3/5.24/4.4, 5.24/4.4/5.28/4.5, 5.28/4.5/5.32/4.6, 5.32/4.6/5.36/4.6, 5.36/4.6/5.4/4.6, 5.4/4.6/5.44/4.6, 5.44/4.6/5.48/4.7, 5.48/4.7/5.52/4.7, 5.52/4.7/5.56/4.8, 5.56/4.8/5.6/4.8, 5.6/4.8/5.64/4.8, 5.64/4.8/5.68/4.8, 5.68/4.8/5.72/4.8, 5.72/4.8/5.76/4.9, 5.76/4.9/5.8/5.0, 5.8/5.0/5.84/5.0, 5.84/5.0/5.88/5.0, 5.88/5.0/5.92/5.0, 5.92/5.0/5.96/5.0}{
	\draw[-, color=red, line width=0.5mm] (\i,\j) -- (\k,\l);
	}
	
	\draw[-, color=black, line width=0.5mm, dashed] (2,0) -- (2,7);
	\draw[-, color=black, line width=0.5mm, dashed] (6,0) -- (6,7);
	
	\foreach \i in {0.4,0.8,1.2,1.6}{
	\foreach \j in {0.4,0.8,...,7.2}{
	\filldraw[black]  (\i,\j) circle (1pt);
	}
	}
	
	\foreach \i in {6.4,6.8,7.2,7.6}{
	\foreach \j in {0.4,0.8,...,7.2}{
	\filldraw[black]  (\i,\j) circle (1pt);
	}
	}
	
	\coordinate[label=below:$a$] (x) at (2,0);
	\draw[-] (2,-0.1) -- (2,0.1);
	\coordinate[label=below:$b$] (x) at (6,0);
	\draw[-] (6,-0.1) -- (6,0.1);

\end{tikzpicture}
\caption{Two paths in $[a,b]$ distributed according to $\mu$. The endpoints of the paths are random and depend on the configuration of the $\omij$'s outside of the strip, represented as black dots here.}
\label{mu paths figure}
\end{figure}
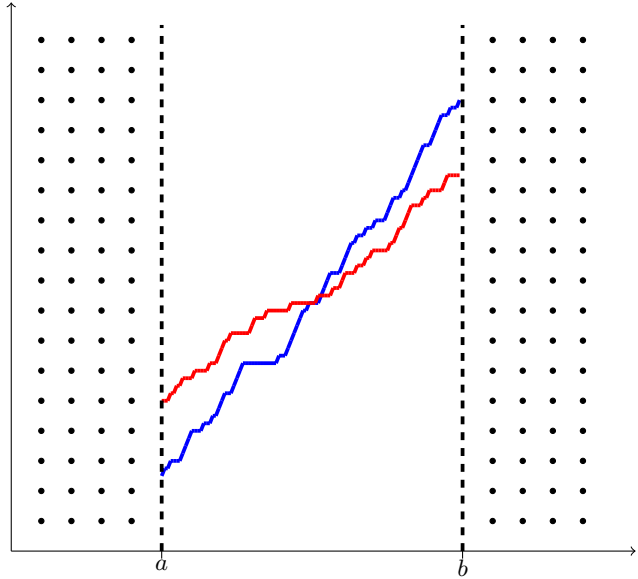

It turns out that the number of intersections of independent Bernoulli bridges behaves very similarly to the number of intersections of Bernoulli random walks, as long as the start and end points are not too far from each other. For instance, in the worst possible case, the starting point is $(a,0)$ and the endpoint is $(b,0)$ or $(b,n_0)$. In both cases there is exactly one path with those endpoints, and so two bridges with those endpoints will intersect everywhere (so $n_0$ times), whereas we want the number of intersections to be of order $\sqrt{n_0}$. So $n_0$ should be small, but not too small, otherwise there will be too many paths on $[a,b]$ with a slope close to 0 or 1. Then all of the calculations done previously will work in this setup.

\section{Local fluctuations of the polymer}
For two integers $0 \le a<b$, a \textit{path on $[a,b]$} is the restriction of a path to the interval $[a,b]$. We denote the set of all paths on $[a,b]$ by $\Pi[a,b]$. Note that if $\pi \in \Pi[a,b]$, then we must have $0 \le \pi(a) \le a$ and $\pi(a) \le \pi(b) \le b$.

For $0 \le x \le a$ and $p \in [0,1]$, we denote by $\P^{[a,b]}_{p,x}$ the usual Bernoulli walk measure on $\Pi[a,b]$ with mean $p$, started at $x$:
$$\P^{[a,b]}_{p,x}(\pi)=\binom{b-a}{\pi(b)-x}p^{\pi(b)-x} (1-p)^{b-a-(\pi(b)-x)}.$$
If we are also given a $y$ such that $x \le y \le x+b-a$, we denote by $\P^{[a,b]}_{x,y}$ the uniform measure on the sets of paths such that $\pi(a)=x$ and $\pi(b)=y$, and we refer to this measure as a \textit{(Bernoulli) bridge}. The corresponding expectations for these measures are denoted by $\E^{[a,b]}_{p,x}$ and $\E^{[a,b]}_{x,y}$. We also use the same notation for the product measure and expectation on $\Pi[a,b]^k$, where the subscripts $x$ and $y$ are then understood to be vectors of size $k$ corresponding to the endpoints of each of the paths, and $p$ is also a vector of size $k$ whose entries are the means of the $k$ Bernoulli walks. We will usually omit the superscript $[a,b]$ for both the Bernoulli walk and bridge measure when $a$ and $b$ have been fixed earlier on.

If $\pi \sim \P^{[a,b]}_{p,x}$, then $p$ will be referred to as the \textit{slope} of $\pi$, and if $\pi \sim \P^{[a,b]}_{x,y}$, the slope of $\pi$ will refer to the slope of the line joining $(a,x)$ to $(b,y)$, namely $(y-x)/(b-a)$.

Given a set of paths $\pi_1, \dots, \pi_k$ in $\Pi[a,b]$, we let
$$v(i,j)=\text{the number of paths that visited the site } (i,j),$$
and 
$$\mathcal{V}(\pi_1, \dots, \pi_k)=\{(i,j): v(i,j) \ge 2\}.$$
The \textit{intersection local time} of the paths $\pi_1, \dots, \pi_k$ is the cardinality of the set $\mathcal{V}(\pi_1, \dots, \pi_k)$, and we denote it by $L(\pi_1,\dots,\pi_k)$. We have the bound
\begin{equation}\label{basic inequality for L}
L(\pi_1, \dots, \pi_k) \le \sum_{1\le j< i \le k} L(\pi_i,\pi_j),
\end{equation}
since the right-hand side counts multiple times all the instances where three or more paths intersect all at once (for example, if $\pi_1,\pi_2,\pi_3$ intersect all at the same time, then the left-hand side counts this intersection once, whereas the right-hand side counts it three times).

\begin{theo}\label{centred moments Z}
Let $\mu$ be a probability measure on paths in $[a,b]$, $n_0=b-a$, and $k$ a positive even integer. Then
$$\E((Z_{\mu}(\b)-1)^k) \le \left(\sum_{j=1}^{n_0} (C_k)^j \b^j \left[\frac{\E_{\mu}(L(\pi_1,\dots,\pi_k)^{kj/2})}{(kj/2)!}\right]^{1/k}\right)^k,$$
where $\pi_1, \dots, \pi_k$ are independent paths distributed according to $\mu$.
\end{theo}

\begin{remark}
As the subscript suggest, the constants $C_k$ depend on $k$, but not $n_0$. They also depend on the collection of weights, but only in terms of the constants involved in Conditions 3 and 4 of valid sets of weights. As discussed in Remark \ref{combination of valid weights}, any combination of two valid sets of weights is again valid, and it's clear that the constants in Conditions 3 and 4 can be chosen so that they work for any possible combination. This means that the constants in the above theorem really only depend on $k$.
\end{remark}

\begin{proof}
For each $i$ and $j$, let $\zeta_{i,j}(\b)=\omij(\b)-1$. Then the $\zeta_{i,j}$ have mean 0, and by Condition 3 of valid sets of weights, $\E|\zeta_{i,j}(\b)|^k \le C_k \b^k$. We can thus rewrite $Z_{\mu}(\b)-1$ as follows
\begin{eqnarray*}
Z_{\mu}(\b)-1 &=& \sum_{\pi \in \Pi[a,b]} \left[ \prod_{i=a}^b (1+\zeta_{i,\pi(i)}(\b))-1\right]\mu(\pi)\\
&=&\sum_{\pi \in \Pi[a,b]} \mu(\pi)\sum_{S \subseteq [a,b], S\neq \emptyset}\prod_{i \in S} \zeta_{i, \pi(i)}(\b).
\end{eqnarray*}
Here the inner sum is taken over non-empty subsets of $[a,b]$ (we interpret intervals in this proof as intervals of integers, so $[a,b]=\{a,a+1,\dots,b\}$). For $1 \le j \le n_0$, let
$$Z_{\mu,j}(\b)=\sum_{\pi \in \Pi[a,b]} \mu(\pi)\sum_{S \subseteq [a,b], |S|=j} \prod_{i \in S} \zeta_{i, \pi(i)}(\b).$$
Then $Z_{\mu}-1=Z_{\mu,1}(\b)+\dots+Z_{\mu,n_0}(\b)$. We will estimate the $k$-th moment of each $Z_{\mu,j}(\b)$.

We have
\begin{equation}\label{k-th moment of Z mu j}
\E(Z_{\mu,j})^k
=\E\left(  \sum_{\pi_1, \dots, \pi_k} \mu(\pi_1) \cdots \mu(\pi_k) \sum_{\substack{S_1, \dots, S_k \\|S_i|=j}}\prod_{i_1 \in S_1} \zeta_{i_1,\pi_1(i_1)} \cdots \prod_{i_k \in S_k}\zeta_{i_k,\pi_k(i_k)}\right).
\end{equation}
Next, by H\"older's inequality and the inequalities $\E|\zeta_{i,j}(\b)|^k \le C_k \b^k$,
\begin{eqnarray*}
&&\E\left|\prod_{i_1 \in S_1} \zeta_{i_1,\pi_1(i_1)} \cdots \prod_{i_k \in S_k}\zeta_{i_k,\pi_k(i_k)}\right|\\
&\le& \left(\E\left(\prod_{i_1 \in S_1} \zeta_{i_1,\pi_1(i_1)}^k\right)\right)^{1/k} \cdots \left(\E\left(\prod_{i_k \in S_k} \zeta_{i_k,\pi_k(i_k)}^k\right)\right)^{1/k}\\
&=&\left(\prod_{i_1 \in S_1} \E(\zeta_{i_1, \pi_1(i_1)}^k)\right)^{1/k} \cdots \left(\prod_{i_k \in S_k} \E(\zeta_{i_k, \pi_k(i_k)}^k)\right)^{1/k}\le (C_k)^j \b^{kj}.
\end{eqnarray*}
Now each product in the inner sum of \eqref{k-th moment of Z mu j} that has non-zero expectation is obtained by taking points $(i_1,j_1), \dots (i_{\ell},j_{\ell})$ in $\mathcal{V}(\pi_1, \dots, \pi_k)$ such that
$$v(i_1,j_1)+\dots+v(i_{\ell},j_{\ell})=kj$$
and multiplying their corresponding weights $\zeta_{i_1,j_1}^{v(i_1,j_i)}, \dots, \zeta_{i_{\ell},j_{\ell}}^{v(i_{\ell},j_{\ell})}$. Since each of those $v(i,j)$ are at least 2, the number of such configurations is at most 
$$\binom{L(\pi_1,\dots,\pi_k)}{\frac{kj}{2}} \le \frac{L(\pi_1,\dots,\pi_k)^{kj/2}}{(kj/2)!}.$$
Hence, the expectation of the inner sum is no larger than
$$\frac{L(\pi_1,\dots,\pi_k)^{kj/2}(C_k)^j \b^{kj}}{(kj/2)!}$$
and therefore
\begin{eqnarray*}
\E(Z_{\mu,j})^k &\le& (C_k)^j \b^{kj}\sum_{\pi_1, \dots, \pi_k} \mu(\pi_1) \cdots \mu(\pi_k)\frac{L(\pi_1,\dots,\pi_k)^{kj/2}}{(kj/2)!}\\
&=&\frac{(C_k)^j \b^{kj}}{(kj/2)!}\E_{\mu}(L(\pi_1,\dots,\pi_k)^{kj/2}).
\end{eqnarray*}
Finally, by the triangle inequality for $L^k$ norms,
\begin{eqnarray*}
\E(Z_{\mu}(\b)-1)^k&=&\E(Z_{\mu,1}(\b)+\dots+Z_{\mu,n_0}(\b))^k \\
&\le& \left((\E(Z_{\mu,1}(\b))^k)^{1/k}+\dots+(\E(Z_{\mu,n_0}(\b))^k)^{1/k}\right)^k\\
&\le& \left(\sum_{j=1}^{n_0} (C_k)^j \b^j \left[\frac{\E_{\mu}(L(\pi_1,\dots,\pi_k)^{kj/2})}{(kj/2)!}\right]^{1/k}\right)^k. 
\end{eqnarray*}
\end{proof}

\section{Intersections of Bernoulli walks and bridges}
We now need to estimate the moments of $L(\pi_1, \dots, \pi_k)$ where the $\pi_j$'s are independent and distributed according to Bernoulli bridge measure. It turns out to be easier to estimate those moments when the $\pi_j$'s are distributed according to Bernoulli walk measure, so the first step is to show that we can replace bridges by walks. The next lemma describes how to make this change.

\begin{lem}\label{bridge replace}
There exists an $n^*$ such that the following holds. For any interval $[a,b]$ of length $n \ge n^*$, for any $\pi$ distributed according to bridge measure on $[a,b]$ with slope $p$, for any path $\widetilde{\pi} \in \Pi[a,(b-a)/2]$ and for any $p \in [1/4,3/4]$,
$$\P^{[a,b]}_{x,y}(\pi|_{[a,(b-a)/2]}=\widetilde{\pi}) \le 2 \P^{[a,b]}_{p,x}(\pi|_{[a,(b-a)/2]}=\widetilde{\pi}).$$
\end{lem}

Here $\pi|_{[a,(b-a)/2]}$ denotes the restriction of $\pi$ to the interval $[a,(b-a)/2]$. Thus Lemma \ref{bridge replace} asserts that the ratio of the probabilities under either measure that $\pi$ takes on a specific trajectory in the first half of $[a,b]$ is bounded by 2 (the actual constant is irrelevant, what matters is that this constant is independent of $p$, $[a,b]$, etc.). This implies that if $F$ is a measurable function on $\Pi[a,b]^k$ which only depends on the behaviour of the paths in the first half of the interval, then
\begin{equation}\label{bridge replace inequality}
\E^{[a,b]}_{x, y}(F(\pi_1, \dots, \pi_k)) \le 2^k \E^{[a,b]}_{p,x}(F(\pi_1, \dots, \pi_k)),
\end{equation}
provided that all the slopes $p_i=(y_i-x_i)/(b-a)$ are in $[1/4,3/4]$. In particular we can take $F$ to be some moment of the number of intersections of the paths in the first half of the interval.

The main ingredient for the proof of Lemma \ref{bridge replace} is the following theorem (\cite[Equation 25]{pit97}).

\begin{theo}[Platonov]
Let $f$ be the probability mass function of the binomial $(n,p)$ distribution, and let $\phi$ be the density of the standard normal distribution. There exists a universal constant $C$, which does not depend on $p$ or $n$, such that
$$\max_{0 \le k \le n} \left|\sqrt{np(1-p)}f(k)-\phi\left(\frac{k-np}{\sqrt{np(1-p)}}\right)\right|\le \frac{C}{\sqrt{np(1-p)}}.$$
\end{theo}

Platonov's theorem gives an explicit bound for the error term in the local central limit theorem that is uniform over all possible choices of parameters of the binomial distribution.

\begin{proof}[Proof of Lemma \ref{bridge replace}]
Without loss of generality, we can assume that $a=0$ and the starting point $x$ is 0. We thus omit the subscripts and superscripts for the probability measures in the rest of the proof except for the slope $p$ and the endpoint of the bridge $y$ (so $\P_p$ is Bernoulli walk measure and $\P_{p,y}$ is bridge measure). Note that $p=y/n$.

Now, we can rewrite what $\P_{p,y}(\pi|_{[0,n/2]}=\widetilde{\pi})$ is in terms of $\P_{p}(\pi|_{[0,n/2]}=\widetilde{\pi})$ as follows:
$$\P_{p,y}(\pi|_{[0,n/2]}=\widetilde{\pi})=\frac{\P_p(\pi|_{[0,n/2]}=\widetilde{\pi})\P_{p,\widetilde{\pi}(n/2)}(\pi(n)=y)}{\P_p(\pi(n)=y)}.$$

By Platonov's theorem,
\begin{eqnarray*}
\P_{p,\widetilde{\pi}(n/2)}(\pi(n)=y) &\le& \frac{1}{\sqrt{2\pi (n-\pi(n/2)) p(1-p)}}+\frac{C}{np(1-p)}\\
&\le& \frac{1}{\sqrt{\pi n p(1-p)}}+\frac{C}{np(1-p)}.
\end{eqnarray*}
Here we used the fact that $\phi(x) \le 1/\sqrt{2\pi}$ for all $x$ and $\pi(n/2)$ can only be at most $n/2$. Also, by Platonov's theorem again,
\begin{eqnarray*}
\P_p(\pi(n)=y)&\ge& \frac{1}{\sqrt{np(1-p)}}\,\phi\left(\frac{y-np}{\sqrt{np(1-p)}}\right)-\frac{C}{np(1-p)}\\
&=&\frac{1}{\sqrt{2\pi n p(1-p)}}-\frac{C}{np(1-p)}.
\end{eqnarray*}
Putting these two inequalities together yields
$$\frac{\P_{p,\widetilde{\pi}(n/2)}(\pi(n)=y)}{\P_p(\pi(n)=y)}\le \frac{\frac{1}{\sqrt{\pi n p(1-p)}}+\frac{C}{np(1-p)}}{\frac{1}{\sqrt{2\pi n p(1-p)}}-\frac{C}{np(1-p)}}=\frac{\sqrt{2np(1-p)}+C}{\sqrt{np(1-p)}-C}.$$
If $n$ is sufficiently large, then this last bound is $\le 2$ uniformly for $p \in [1/4,3/4]$.
\end{proof}
A very similar argument also works if one considers instead the behaviour of paths in the second half of the interval $[a,b]$, so Lemma \ref{bridge replace} also holds with a path $\widetilde{\pi}$ on $[(b-a)/2,b]$. If we denote by $L(\pi_1, \dots, \pi_k;c,d)$ the number of intersections of the paths $\pi_1, \dots, \pi_k$ in a subinterval $[c,d] \subseteq [a,b]$, then it follows from \eqref{bridge replace inequality} and the triangle inequality for $L^s$ norms that for all $s \ge 1$,
\begin{equation}\label{bridge replace inequality for L}
\begin{split}
&\E^{[a,b]}_{x, y}(L(\pi_1, \dots, \pi_k)^s)\\
&\le \E^{[a,b]}_{x, y}(L(\pi_1, \dots, \pi_k;a,(b-a)/2)+L(\pi_1, \dots, \pi_k;(b-a)/2,b))^s\\
&\le ([\E^{[a,b]}_{x, y}(L(\pi_1, \dots, \pi_k;a,(b-a)/2))^s]^{1/s}\\
&\phantom{\le}+[\E^{[a,b]}_{x, y}(L(\pi_1, \dots, \pi_k;(b-a)/2,b))^s]^{1/s})^s\\
&\le ([2^k\E^{[a,b]}_{p, x}(L(\pi_1, \dots, \pi_k;a,(b-a)/2))^s]^{1/s}\\
&\phantom{\le}+[2^k\E^{[a,b]}_{p, x}(L(\pi_1, \dots, \pi_k;(b-a)/2,b))^s]^{1/s})^s\\
&\le 2^{k+s} \E^{[a,b]}_{p,x}(L(\pi_1, \dots, \pi_k))^s.
\end{split}
\end{equation}
In view of \eqref{basic inequality for L}, we will only really need to consider the case $k=2$.

So now we need to estimate the moments of the number of intersections of two Bernoulli walks. There is yet another simplification we can make. Extend the walks to infinite Bernoulli walks, and let $G(\pi_1,\pi_2)$ be the number of times the two walks intersect before there is an interval of length at least $b-a$ where they do not intersect. Then clearly we have $L(\pi_1,\pi_2; a, b) \le G(\pi_1,\pi_2)$, since any intersection that happens in $[a,b]$ is an intersection that happens before there is an interval of length at least $b-a$ with no intersection. The advantage of working with $G(\pi_1,\pi_2)$ instead is that its distribution is much simpler; it is the $\Geo_0(q_{b-a})$, where $q_k$ is the probability that the two paths do not intersect before time $k$ (we use $\Geo_0$ to denote the geometric distribution on $\{0,1,2,\dots\}$). If $X \sim \Geo_0(q)$, then its moments satisfy
\begin{equation}\label{geometric moments}
\E(X^m) \le \frac{C m!}{q^m}
\end{equation}
for all $q \le 1/2$. One way to see this is that if $Y \sim \Exp(1)$, then $X=\lfloor -Y/\log (1-q) \rfloor$ follows the $\Geo_0(q)$ distribution, $\E(Y^m)=m!$ by integration by parts, and $-\log(1-q) \sim q$ as $q \to 0$. So all we need is a lower bound on the probability that the first intersection of two Bernoulli walks happens after time $k$.

\begin{lem}\label{lower bound intersection prob}
Let $\pi_1, \pi_2$ be independent Bernoulli walks with slopes $p_1, p_2$ respectively, started at the same point, and let $T=\min\{k \ge 1: \pi_1(k)=\pi_2(k)\}$. Then there is an absolute constant $C$ such that for all $n \ge 1$,
$$\P(T \ge n) \ge \frac{C(p_1(1-p_2)+(1-p_1)p_2)}{\sqrt{n}}.$$
\end{lem}

\begin{proof}
Let $X(k)=\pi_1(k)-\pi_2(k)$. Then $X$ is a ``lazy'' random walk, with transition probabilities given by
$$X(k+1)-X(k)=
\begin{cases}
0 & \quad \text{with probability } p_1 p_2+(1-p_1)(1-p_2)\\
1 & \quad \text{with probability } p_1 (1-p_2)\\
-1 & \quad \text{with probability } (1-p_1)p_2.
\end{cases}
$$
The law of $X$ is the same as a random walk $Y$ with transition steps given by
$$Y(k+1)-Y(k)=
\begin{cases}
1 & \quad \text{with probability } \frac{p_1 (1-p_2)}{p_1(1-p_2)+(1-p_1)p_2}\\
-1 & \quad \text{with probability } \frac{(1-p_1)p_2}{p_1(1-p_2)+(1-p_1)p_2}
\end{cases}
$$
and where at each step, the walk stays there for a $\text{Geo}_0(p_1(1-p_2)+(1-p_1)p_2)$ amount of time before jumping according to $Y$. Let $T'=\min\{k \ge 1: Y(k)=0\}$ be the first return to 0 of $Y$. If the walk $X$ immediately jumps away from 0 (i.e the geometric time spent at 0 is 0) and $T' \ge n$, then we certainly have $T \ge n$. Indeed it will take longer for $X$ to come back to 0 than $Y$ because after each time it jumps, it needs to stay at the same position for a geometric amount of time. We thus have
$$\P(T \ge n) \ge \P(T' \ge n,X(1) \neq 0)= (p_1(1-p_2)+(1-p_1)p_2)\P(T' \ge n).$$

There is an explicit formula for the law of the first return time to 0 of an asymmetric simple random walk $S$ started at 0 with probability $q$ of going up and probability $1-q$ of going down; it is given by

$$\P(S(1)\neq 0, \dots, S(2k-1) \neq 0, S(2k)=0)=[q(1-q)]^k \frac{\binom{2k}{k}}{2k-1}.$$

See \cite[Chapter XIV, Section 9, Exercise 13]{feller68}. For any $k$, this quantity will clearly be the largest when $q=1/2$, so it follows from Stirling's formula that

\begin{eqnarray*}
\P(T' \ge n)&=&1-\sum_{k=0}^{\lfloor n/2 \rfloor} [q(1-q)]^k \frac{\binom{2k}{k}}{2k-1} \ge 1-\sum_{k=0}^{\lfloor n/2 \rfloor} \frac{\binom{2k}{k}}{4^k(2k-1)}\\
&=&\sum_{k=\lfloor n/2 \rfloor}^{\infty}\frac{\binom{2k}{k}}{4^k(2k-1)} \ge \frac{C}{\sqrt{n}}
\end{eqnarray*}
for some appropriate constant $C$.
\end{proof}
Lemma \ref{lower bound intersection prob} thus gives a lower bound which is \textit{uniform} over all slopes $p_1,p_2 \in [1/4,3/4]$. Combining this with \eqref{geometric moments}, the bound $L(\pi_1,\pi_2) \le G(\pi_1,\pi_2)$ and \eqref{bridge replace inequality for L}, we obtain
\begin{equation}\label{moments of intersections bound}
\E_{x,y}(L(\pi_1,\pi_2)^m) \le C m! A^m n^{m/2}
\end{equation}
for constants $C$ and $A$ that do not depend on $n$ and are uniform over all slopes $(y_1-x_1)/(b-a)$ and $(y_2-x_2)/(b-a)$ in $[1/4,3/4]$. 

We finish this section with two technical lemmas where we prove an analogous bound to \eqref{moments of intersections bound} when $m=2$ for triplets of paths, as well as an estimate for the moment generating function of $L(\pi_1,\pi_2)$. These will be needed for the next section.

\begin{lem}\label{triple intersection lemma}
Let $\pi_1, \pi_2, \pi_3$ be independent Bernoulli bridges on $[0,n]$ of slopes $p_1, p_2, p_3$ respectively, and let 
$$L_3(\pi_1,\pi_2,\pi_3)=\sum_{k=1}^n \mathbbm{1}_{\{\pi_1(k)=\pi_2(k)=\pi_3(k)\}}$$
be the number of times that all three paths intersect simultaneously before time $n$. Then there is a constant $C$ such that for all $p_1,p_2,p_3 \in [1/4,3/4]$, and for all $n$ sufficiently large,
$$\E(L_3(\pi_2,\pi_2,\pi_3)^2) \le C (\log n)^2.$$
\end{lem}

\begin{proof}
By a similar argument to \eqref{bridge replace inequality for L} with $L$ replaced by $L_3$, we can replace ``Bernoulli bridges'' with ``Bernoulli walks'' with the same slopes. Assume first that all three paths start at the same position (without loss of generality this initial position is 0). Let
$$q_k=\P(\pi_1(k)=\pi_2(k)=\pi_3(k)).$$

We first show that $q_n=O(1/n)$. For each $i$, let $f_i$ be the probability mass function of $\pi_i$ at time $n$. Then 
\begin{equation}\label{prob of triple intersection}
q_n=\sum_{k=0}^n f_1(k)f_2(k)f_3(k).
\end{equation}

By Platonov's theorem, we have, for each $i$,
$$f_i(k) \le \frac{1}{\sqrt{n p_i (1-p_i)}} \phi\left(\frac{k-np_i}{\sqrt{np_i(1-p_i)}}\right)+\frac{C}{np_i(1-p_i)},$$

so after expanding the product in each term of \eqref{prob of triple intersection}, the main order term is 
\begin{eqnarray*}
&\le& \frac{C}{n^{3/2}} \sum_{k=0}^n \phi\left(\frac{k-np_1}{\sqrt{np_1(1-p_1)}}\right) \phi\left(\frac{k-np_2}{\sqrt{np_2(1-p_2)}}\right)\phi\left(\frac{k-np_3}{\sqrt{np_3(1-p_3)}}\right)\\
&\le& \frac{C}{n} \int_{-\sqrt{n}}^{\sqrt{n}} \phi \left( \frac{x-p_1 \sqrt{n}}{\sqrt{p_1(1-p_1)}}\right) \phi \left( \frac{x-p_2 \sqrt{n}}{\sqrt{p_2(1-p_2)}}\right)\phi \left( \frac{x-p_3 \sqrt{n}}{\sqrt{p_3(1-p_3)}}\right) dx\\
&\phantom{=}&+ O\left( \frac{1}{n^{3/2}} \right)\\
&\le& \frac{C}{n}.
\end{eqnarray*}
Here we used the fact that the sum can be interpreted as a Riemann sum, and that the difference between a Riemann sum and the integral it approximates is at most a constant times the modulus of continuity of the integrand times the size of the mesh. In this case the size of the mesh is $1/\sqrt{n}$, and the modulus of continuity is bounded by a constant independent of $n$ given that $p_1,p_2,p_3 \in [1/4,3/4]$. The integral is easily seen to be bounded, using say H\"older's inequality. The lower order terms can be treated similarly: instead of having products of three Gaussian densities, we'll have products of two, one or zero Gaussian densities along with a higher exponent of $1/n$, and in all cases after we approximate the sum with an integral, the result will be $O(1/n^{3/2})$.

We then have for $n \ge 2$,
\begin{eqnarray*}
\E(L_3(\pi_1,\pi_2,\pi_3)^2)&=&\E \left(\sum_{k=1}^n \mathbbm{1}_{\{\pi_1(k)=\pi_2(k)=\pi_3(k)\}}\right)^2\\
&=&\E\left(2\sum_{1\le j<k \le n} \mathbbm{1}_{\{\pi_1(j)=\pi_2(j)=\pi_3(j)\}}\mathbbm{1}_{\{\pi_1(k)=\pi_2(k)=\pi_3(k)\}}\right.\\
&\phantom{=}&+\left. \sum_{k=1}^n \mathbbm{1}_{\{\pi_1(k)=\pi_2(k)=\pi_3(k)\}}\right)\\
&=&2\sum_{1 \le j<k \le n} q_j q_{k-j}+\sum_{k=1}^n q_k\\
&\le& C\left(\sum_{j=1}^{n-1} \frac{1}{j} \sum_{k=j+1}^n \frac{1}{k-j}+\sum_{k=1}^n \frac{1}{k}\right) \le C(\log n)^2.
\end{eqnarray*}

Finally, if the paths do not start at the same location, then we still have
$$L_3(\pi_1,\pi_2,\pi_3) \le 1+L_3(\pi_1',\pi_2',\pi_3')$$
where the $\pi_i'$ are the paths after the first intersection. The above bounds then holds for $L_3(\pi_1',\pi_2',\pi_3')$ and the required inequality follows easily.
\end{proof}

\begin{lem}\label{mgf lemma}
Let $\pi_1$ and $\pi_2$ be independent Bernoulli bridges on $[0,n]$ of slopes $p_1$ and $p_2$ respectively, and let $A,\delta>0$. Then there is a constant $C$ such that for all $p_1,p_2 \in [1/4,3/4]$,
$$\E\left[\left(1+\frac{A}{n^{1/2+\delta}}\right)^{6L(\pi_1,\pi_2)}\right] \le C$$
\end{lem}

\begin{proof}
We can again use an argument similar to \eqref{bridge replace inequality for L} to replace bridges with walks of the same slopes, this time using convexity of the function $x \mapsto (1+A/n^{1/2+\delta})^{12x}$ instead of the triangle inequality. This will change the exponent from 6 to 12. We once again use the inequality $L(\pi_1,\pi_2) \le G(\pi_1,\pi_2)$, where $G(\pi_1,\pi_2)$ is the number of times $\pi_1$ and $\pi_2$ intersect before there is an interval of length at least $n$ on which they do not intersect. Then $G(\pi_1,\pi_2)$ follows a geometric distribution, and by Lemma \ref{lower bound intersection prob}, its parameter is at least $C/\sqrt{n}$. Since the $\Geo_0(q)$ stochastically dominates the $\Geo_0(p)$ whenever $q<p$, we may as well replace $G(\pi_1,\pi_2)$ by $X$ where $X \sim \Geo_0(C/\sqrt{n})$. The expectation is then bounded by the moment generating function of a geometric distribution, which we can explicitly compute:
\begin{eqnarray*}
\E\left[\left(1+\frac{A}{n^{1/2+\delta}}\right)^{12X}\right]&=&\sum_{j=0}^{\infty} \left(1+\frac{A}{n^{1/2+\delta}}\right)^{12j} \frac{C}{\sqrt{n}} \left(1-\frac{C}{\sqrt{n}}\right)^j\\
&=&\frac{C}{\sqrt{n}(1-(1+A/n^{1/2+\delta})^{12}(1-C/\sqrt{n}))}.
\end{eqnarray*}
Note that 
\begin{eqnarray*}
\left(1+\frac{A}{n^{1/2+\delta}}\right)^{12} \left(1-\frac{C}{\sqrt{n}}\right)&=&\left(1+o\left(\frac{1}{\sqrt{n}}\right)\right)\left(1-\frac{C}{\sqrt{n}}\right)\\
&=&1-\frac{C}{\sqrt{n}}+o\left(\frac{1}{\sqrt{n}}\right).
\end{eqnarray*}
This shows that if $n$ is large enough, then $(1+A/n^{1/2+\delta})^{12} (1-C/\sqrt{n})<1$, and the above series does indeed converge. It also follows from this calculation that
$$\frac{C}{\sqrt{n}(1-(1+A/n^{1/2+\delta})^{12}(1-C/\sqrt{n}))}=\frac{C}{C+o(1)},$$
so this quantity is bounded, and the lemma follows.
\end{proof}

\section{Proof of Theorem \ref{perturbation thm}}\label{perturbation proof section}
Let $[a,b] \subset [0,2n]$ be some interval of length $n_0$, $S$ a set of paths in $[a,b]$, and let $\widetilde{S}$ be the set of paths on $[0,2n]$ whose restriction to $[a,b]$ is in $S$. We wish to obtain an estimate for $\mu(S)$  (recall $\mu$ was defined in \eqref{mu definition}) in terms of $\P_{0,2n}(\widetilde{S})$ (recall in our notation, $\P_{0,2n}$ is bridge measure on paths started at 0 and ending at $2n$). By Property 4 of valid sets of weights we have
$$\P(e^{-c_1 \b^s} \le \omij(\b) \le e^{c_2 \b^s}) \ge 1-A(s) e^{-c_3\b^{s-1}}$$
for all $i$ and $j$. Letting $\Wc_s$ be the event that all $\omij(\b)$ satisfy 
$$e^{-c_1 \b^s} \le \omij(\b) \le e^{c_2 \b^s},$$ 
we then have by a simple union bound that
$$\P(\Wc_s)\ge 1-An^2 e^{-c_3 \b^{s-1}}.$$

Then on the event $\Wc_s$, the random probability measure $\mu$ satisfies
\begin{equation}\label{inequality for mu}
\mu(S)=\frac{\sum_{\pi \in \widetilde{S}} \prod_{\ell \notin [a,b]} \om_{\ell,\pi(\ell)}(\b)}{\sum_{\pi} \prod_{\ell \notin [a,b]} \om_{\ell,\pi(\ell)}(\b)}\le \frac{e^{c_2 \b^s(2n-n_0)} |\widetilde{S}|}{e^{-c_1 \b^s (2n-n_0)} \binom{2n}{n}} \le e^{C \b^s n} \P_{0,2n}(\widetilde{S}).
\end{equation}

We will take as $S$ the set of paths whose slope on $[a,b]$ is not in $[1/4,3/4]$:
$$S=\left\{\pi:  \frac{\pi(b)-\pi(a)}{b-a} \notin \left[ \frac{1}{4},\frac{3}{4} \right]\right\}=\left\{\pi: \left|\pi(b)-\pi(a)-\frac{n_0}{2}\right| \ge \frac{n_0}{4}\right\}.$$

Then with $\P_{1/2,0}$ denoting the usual Bernoulli random walk measure on $[0,2n]$ started at 0 with mean $1/2$, we have by Stirling's formula and Hoeffding's inequality,
$$\P_{0,2n}(\widetilde{S})=\frac{|\widetilde{S} \cap \{\pi: \pi(2n)=n\}|}{\binom{2n}{n}}\le C\sqrt{n}\,\P_{1/2,0}(\widetilde{S}) \le C\sqrt{n}e^{-Cn_0},$$
and so together with \eqref{inequality for mu},
\begin{equation}\label{bound for mu of S}
\mu(S) \le C\sqrt{n}e^{C (\b^s n-n_0)}
\end{equation}
on $\Wc_s$, with $C$ independent of $s$ and $n$. Now, $\b^s n=n^{1-s\alpha}$ and $n_0=\b^{-4/(1+4\delta)}=n^{4\alpha/(1+4\delta)}$, so as long as
$$1-s\alpha<\frac{4\alpha}{1+4\delta},$$
$\mu(S)$ will be exponentially decreasing in some power of $n$. If $\alpha>1/5$, then we can find a $\delta>0$ and a $0<s<1$ such that the above inequality holds. To summarize, we have shown that the probability of obtaining a configuration of weights for which the random measure $\mu$ assigns exponentially small probability to paths with slopes not in $[1/4,3/4]$ is exponentially close to 1. We assume for the rest that $s$ has been fixed to satisfy the above, and we omit the subscript in $\Wc_s$.

We now estimate the moments of $L(\pi_1,\dots,\pi_k)$ for $\pi_1, \dots, \pi_k$ independent and distributed according to $\mu$. Recall inequality \eqref{basic inequality for L}:
$$L(\pi_1, \dots, \pi_k) \le \sum_{1\le j< i \le k} L(\pi_i,\pi_j).$$
There are $\binom{k}{2}$ terms on the right, and they clearly all have the same distribution, so by the triangle inequality for $L^m$ norms, we have for any $m \ge 1$,
$$
\E_{\mu}(L(\pi_1, \dots, \pi_k)^m) \le \left(\sum_{1 \le j<i \le k} \E_{\mu}(L(\pi_i,\pi_j)^m)^{1/m}\right)^m
=\binom{k}{2}^m\E_{\mu}(L(\pi_1,\pi_2)^m).
$$
To estimate the expectation, we condition on the endpoints of $\pi_1$ and $\pi_2$. Let 
\begin{eqnarray*}
\Gc&=&\{(x,y): (y-x)/(b-a) \in [1/4,3/4]\}\\
\Bc&=&\{(x,y): (y-x)/(b-a) \notin [1/4,3/4]\}.
\end{eqnarray*}
Thus $\Gc$ are the points such that the line joining $(a,x)$ to $(b,y)$ has slope in $[1/4,3/4]$ (the ``good endpoints'') and $\Bc$ are those for which the line has slope not in $[1/4,3/4]$ (the ``bad endpoints''). We therefore have
\begin{eqnarray*}
\E_{\mu}(L(\pi_1,\pi_2)^m)
&=&\sum_{x_1 \in \Bc \text{ or } x_2 \in \Bc} \E_{x_1,x_2}^{[a,b]}(L(\pi_1,\pi_2)^m) \mu(\{\pi_1(a,b)=x_1, \pi_2(a,b)=x_2\})\\
&\phantom{=}&+\sum_{x_1 \in \Gc \text{ and } x_2 \in \Gc} \E_{x_1,x_2}^{[a,b]}(L(\pi_1,\pi_2)^m) \mu(\{\pi_1(a,b)=x_1, \pi_2(a,b)=x_2\})\\
&=&\Sigma_{\Bc}+\Sigma_{\Gc}.
\end{eqnarray*}
For the first term, we use the trivial bound $L(\pi_1,\pi_2) \le n_0$ and the fact that the set of paths with bad endpoints has small $\mu$ measure to obtain
$$\Sigma_{\Bc} \le Cn_0^m e^{-n_0^{\gamma}}$$
for some $\gamma>0$. For the second term, we use \eqref{moments of intersections bound}:
$$\E_{x_1,x_2}^{[a,b]}(L(\pi_1,\pi_2)^m) \le C m! A^m n_0^{m/2}$$
where $C$ and $A$ are constants which do not depend on $x_1$ or $x_2$. It follows that
$$\Sigma_{\Gc}\le C m! A^m n_0^{m/2}$$
and so for $n_0$ sufficiently large,
$$\E_{\mu}(L(\pi_1, \dots, \pi_k)^m)\le \binom{k}{2}^m\E_{\mu}(L(\pi_1,\pi_2)^m)\le C m! (Ak^2)^m n_0^{m/2}.$$
Combining this with Theorem \ref{centred moments Z}, we obtain
\begin{eqnarray*}
\E((Z_{\mu}(\b)-1)^k) &\le& \left(\sum_{j=1}^{n_0} (C_k)^j \b^j \left[\frac{\E_{\mu}(L(\pi_1,\dots,\pi_k)^{kj/2})}{(kj/2)!}\right]^{1/k}\right)^k\\
&\le& \left(\sum_{j=1}^{n_0} (C_k)^j \b^j C^{1/k}(Ak^2)^{j/2} n_0^{j/4} \right)^k \le \left(\frac{C_k \b n_0^{1/4}}{1-C_k \b n_0^{1/4}}\right)^k.
\end{eqnarray*}
We had chosen $n_0$ so that $\b=n_0^{-(1/4+\delta)}$, and so if $n_0$ is large enough, then
\begin{equation}\label{centred moments Z 2}
\E((Z_{\mu}(\b)-1)^k)\le \frac{C_k}{n_0^{k\delta}}.
\end{equation}
Recall that $Z_{\mu}$ was given by $Z_{n,\b}/\Zc$ in terms of the random measure $\mu$:
$$Z_{\mu}(\b)=\sum_{\pi \in \Pi[a,b]} \mu(\pi) \prod_{\ell \in [a,b]} \om_{\ell,\pi(\ell)}(\b)=\frac{Z_{n,\b}}{\Zc}.$$
The constants $C_k$ in \eqref{centred moments Z 2} are uniform over any configuration of weights in $\Wc$.

By \eqref{taylor bound 2}, we need an estimate for $\E(W_n^3/Z_n^3)$. We do so by conditioning on the weights outside of the strip 
$$\Delta=\{(i,j):a \le i \le b, 0 \le j \le 2i\}.$$ 
Note that $\Delta$ is the set of locations in the plane that could be visited by a path. When the weights are not in $\Wc$, then $W_n/Z_n \le 1$, and the probability of getting a configuration not in $\Wc$ is $\le Ce^{-Cn^{\gamma}}$ for some $\gamma>0$. When the weights are in $\Wc$, \eqref{centred moments Z 2} together with Markov's inequality imply that
\begin{equation}\label{conditional expectation of W over Z cubed}
\begin{split}
\E\left( \left. \frac{W_n^3}{Z_n^3} \right| \om \right)&=\E\left( \left. \frac{(W_n/\Zc)^3}{(Z_n/\Zc)^3} \right| \om \right)\\
&=\E\left( \left. \frac{(W_n/\Zc)^3}{(Z_n/\Zc)^3}\mathbbm{1}_{\{|Z_n/\Zc-1|< 1/2\}} \right| \om \right)\\
&\phantom{=}+\E\left( \left. \frac{(W_n/\Zc)^3}{(Z_n/\Zc)^3}\mathbbm{1}_{\{|Z_n/\Zc-1|\ge 1/2\}} \right| \om \right)\\
&\le 8\E\left(\left. \frac{W_n^3}{\Zc^3} \right| \om \right)+\P\left(\left. \left|\frac{Z_n}{\Zc}-1\right| \ge \frac{1}{2}\right| \om \right)\\
&\le 8\E\left(\left. \frac{W_n^3}{\Zc^3} \right| \om \right)+\frac{C_k}{n_0^{k\delta}},
\end{split}
\end{equation}
where we use $\om$ to denote a configuration of the weights outside of $\Delta$, and such that $\om \in \Wc$. This holds for any even $k$, so by picking $k$ large enough, we can make the second term of smaller order than the first.

We can express $W_n/\Zc$ in terms of $\mu$ in a similar way to $Z_n/\Zc$ by only summing over the paths that go through the point $(i,j)$:
\begin{equation}\label{W formula}
\frac{W_n}{\Zc}=\sum_{\pi, (i,j) \in \pi} \mu(\pi)\prod_{\ell \in [a,b]} \om_{\ell, \pi(\ell)}(\b).
\end{equation}
So we now cube this expression, sum over all points $(i,j)$ in the strip $\Delta$ and take expectations to get a bound for \eqref{conditional expectation of W over Z cubed}. There are $O(n \cdot n_0)$ points in $\Delta$, so the second term in \eqref{conditional expectation of W over Z cubed} will be $C_k n/n_0^{k\delta-1}$, which again we can make as small as we want by choosing a large enough $k$. The precise estimate for the first term is contained in the following lemma.

\begin{lem}
There is a constant $C$ such that for any configuration $\om \in \Wc$,
$$\E \left( \left. \sum_{(i,j) \in \Delta}\frac{W_n(i,j)^3}{\Zc^3}\right| \om \right) \le C \log n.$$
\end{lem}
\begin{proof}
Let $\zeta_{i,j}(\b)=\om_{i,j}(\b)-1$. Then using \eqref{W formula}, we get, after summing over $i,j$ in the strip:
\begin{eqnarray*}
&\E& \left( \left. \sum_{(i,j) \in \Delta}\frac{W_n(i,j)^3}{\Zc^3}\right| \om \right)\\
&=&\E \left( \left. \sum_{(i,j) \in \Delta} \left(\sum_{\pi \in \Pi[a,b], (i,j) \in \pi} \mu(\pi)\prod_{\ell \in [a,b]} (1+\zeta_{\ell, \pi(\ell)}(\b)) \right)^3 \right| \om \,\,\right)\\
&=&\E\left( \sum_{(i,j) \in \Delta} \sum_{\pi_1,\pi_2,\pi_3 \in \Pi[a,b]} \mu(\pi_1)\mu(\pi_2)\mu(\pi_3) \mathbbm{1}_{\{(i,j) \in \pi_1,\pi_2,\pi_3\}}\right. \\
&\phantom{=}&\times \left. \left.\prod_{\ell \in [a,b]}(1+\zeta_{\ell,\pi_1(\ell)}(\b))(1+\zeta_{\ell,\pi_2(\ell)}(\b))(1+\zeta_{\ell,\pi_3(\ell)}(\b))\right| \om \right)\\
&=&\E\left( \sum_{\pi_1,\pi_2,\pi_3 \in \Pi[a,b]} \mu(\pi_1)\mu(\pi_2)\mu(\pi_3)L_3(\pi_1,\pi_2,\pi_3)\right.\\
&\phantom{=}&\times \left. \left.\prod_{\ell \in [a,b]}(1+\zeta_{\ell,\pi_1(\ell)}(\b))(1+\zeta_{\ell,\pi_2(\ell)}(\b))(1+\zeta_{\ell,\pi_3(\ell)}(\b))\right| \om \right).
\end{eqnarray*}
We now swap the sum with the expectation. Since the $\zeta_{i,j}$ are independent of each other and of $\om$ and have mean 0, the conditional expectation of the product
\begin{equation}\label{one product term}
\E[(1+\zeta_{\ell,\pi_1(\ell)}(\b))(1+\zeta_{\ell,\pi_2(\ell)}(\b))(1+\zeta_{\ell,\pi_3(\ell)}(\b))\, | \om]
\end{equation}
will be 1 if the paths $\pi_1,\pi_2$ and $\pi_3$ are all at different locations at time $\ell$; it will be $1+\E(\zeta_{\ell,j}(\b)^2)$ if exactly two paths intersect at time $\ell$ (at location $j$) and it will be $1+3\E(\zeta_{\ell,j}(\b)^2)+\E(\zeta_{\ell,j}(\b)^3)$ if all three paths intersect at time $\ell$ (at location $j$). By Condition 3 of valid sets of weights, we have $\E|\zeta_{i,j}^k| \le C_k \b^k$, so the largest that \eqref{one product term} can be if it isn't 1 is $1+C\b^2$. The number of terms which aren't 1 when taking the product over the $\ell \in [a,b]$ is exactly $L(\pi_1,\pi_2,\pi_3)$, and so we obtain
\begin{eqnarray*}
&\E& \left( \left. \sum_{(i,j) \in \Delta}\frac{W_n(i,j)^3}{\Zc^3}\right| \om \right)\\
&\le& \sum_{\pi_1,\pi_2,\pi_3 \in \Pi[a,b]} \mu(\pi_1)\mu(\pi_2)\mu(\pi_3) L_3(\pi_1,\pi_2,\pi_3) (1+C\b^2)^{L(\pi_1,\pi_2,\pi_3)}\\
&=&\E_{\mu}\left(L_3(\pi_1,\pi_2,\pi_3)(1+C\b^2)^{L(\pi_1,\pi_2,\pi_3)}\right).
\end{eqnarray*}
Using the bound $L(\pi_1,\pi_2,\pi_3) \le L(\pi_1,\pi_2)+L(\pi_1,\pi_3)+L(\pi_2,\pi_3)$, that each of these have the same law and H\"older's inequality, we get
\begin{eqnarray*}
&\E_{\mu}&\left(L_3(\pi_1,\pi_2,\pi_3)(1+C\b^2)^{L(\pi_1,\pi_2,\pi_3)}\right)\\
&\le& \E_{\mu} \left(L_3(\pi_1,\pi_2,\pi_3)(1+C\b^2)^{L(\pi_1,\pi_2)+L(\pi_1,\pi_3)+L(\pi_2,\pi_3)}\right)\\
&\le& \left(\E_{\mu}(L_3(\pi_1,\pi_2,\pi_3)^2)\right)^{1/2}\left(\E_{\mu}(1+C\b^2)^{6L(\pi_1,\pi_2)}\right)^{1/2}.
\end{eqnarray*}
Since $\om \in \Wc$, we can replace expectation with respect to the random measure $\mu$ by expectation with respect to bridge measure (by doing a similar breaking down of paths into those with good versus bad endpoints). Since $\b^2=o(1/\sqrt{n_0})$, the exponential term $(1+C\b^2)^{6L(\pi_1,\pi_2)}$ is at worst 
$$(1+C\b^2)^{6n_0} \le Ce^{C\sqrt{n_0}}$$
on the set of paths with bad endpoints, but by \eqref{bound for mu of S}, the $\mu$ measure of those paths with bad endpoints is exponentially decreasing in $n_0$, so this term doesn't cause any problem. After this switch to bridge measure, the first term is $O(\log n)$ by Lemma \ref{triple intersection lemma}, and the second term is $O(1)$ by Lemma \ref{mgf lemma}. 
\end{proof}
We have therefore shown that 
$$\sum_{(i,j) \in \Delta} \E\left( \frac{W_n^3}{Z_n^3} \right) \le C\log n +\frac{C_k}{n_0^{k\delta-2}}.$$
The exact same estimate also holds if we replace $Z_n$ by $W_n+V_n$. Indeed, $W_n+V_n$ is the partition function for the polymer where all the weights are the same except for the $(i,j)$ point where the weight has been replaced by 1. But as discussed in Remark \ref{combination of valid weights}, it's possible to pick the constants in the definition of valid sets of weights so that the inequalities in the definition hold for any combination of weights in a finite collection of sets of weights. All the estimates were obtained in terms of those constants, so the rest of the calculations work for $W_n+V_n$ as well. Therefore plugging these estimates in \eqref{taylor bound 2} and summing over all $(i,j)$ in the strip,
$$
\sum_{(i,j) \in \Delta}\E\left|f\left(\frac{\log(V_n+\omij W_n)-a_n}{\sigma}\right)-f\left(\frac{\log(V_n+\omij' W_n)-a_n}{\sigma}\right)\right|
\le \frac{C \b^3}{\sigma} \log n.
$$
Finally, we break up the entire range $[0,2n]$ into strips of size $n_0$, and sum the errors over all the strips. It should be noted that in all the work that has been done, the location of the strip was irrelevant; what mattered was its width. Thus the same estimates and constants will work for any of the strips. There are $2n/n_0$ different strips, and so the total error is 
\begin{equation}\label{final error estimate}
\E\left|f\left(\frac{\log Z_n-a_n}{\sigma}\right)-f\left(\frac{\log Z'_n-a_n}{\sigma}\right)\right| \le \frac{C\b^3 n \log n}{\sigma n_0}.
\end{equation}
Note that if $2n/n_0$ is not an integer, then we can just add some overlapping strips to cover the whole of $[0,2n]$. Some points may be contained in two different strips, but this just means we add the errors for those points twice, and so all this does is add an extra factor of 2 in the above bound.

We now express $\b, n_0$ and $\sigma$ in terms of $n$ to conclude. We had $\b=n^{-\alpha}$, $n_0=n^{4\alpha/(1+4\delta)}$ and $\sigma=\b^{4/3}n^{1/3}$. So then
$$\frac{\b^3 n\log n}{\sigma n_0}=\frac{\b^3 n\log n}{\b^{4/3}n^{1/3}n_0}=\frac{\b^{5/3}n^{2/3}\log n}{n^{4\alpha/(1+4\delta)}}=n^{\lambda}\log n$$
where
$$\lambda=-\frac{5\alpha}{3}+\frac{2}{3}-\frac{4\alpha}{1+4\delta}=\frac{(2-17\alpha)+8\delta-20\alpha \delta}{3(1+4\delta)}.$$
Since $\alpha>1/5$, we can find a small enough $\delta>0$ such that $\lambda<0$. Then the right-hand side of \eqref{final error estimate} goes to 0 as $n \to \infty$, and this shows that the limiting scaled distributions for $\log Z_n$ and $\log Z'_n$ are the same. This concludes the proof of Theorem \ref{perturbation thm}!

\begin{remark}\label{alpha for higher moments}
The above calculation seems to indicate that Theorem \ref{perturbation thm} should hold for any $\alpha>2/17$. In fact, if we strengthen our hypotheses and assume that the weights $\omij(\b)$ and $\omij'(\b)$ have the same first $k$ moments, then we can do all the calculations of this section with a Taylor expansion of order $k+1$ instead. The end result will be the same estimate as \eqref{final error estimate} but with $\b^3$ replaced by $\b^{k+1}$. After substituting the values of $\b$ and $n_0$ in terms of $n$ in this expression and simplifying, one obtains a bound which goes to 0 provided that
$$\alpha>\frac{2}{3k+11}.$$
This for now is only a conjecture, since we are at the moment limited to $\alpha>1/5$ by \eqref{bound for mu of S}.
\end{remark}

\section{Tracy--Widom fluctuations}\label{Tracy-Widom fluctuations section}

In order to apply Theorem \ref{perturbation thm} to prove Theorem \ref{TW thm}, we need to compare the moments of the weights $e^{\b \xi_{i,j}}$ with some valid set of weights $\omij(\b)$ for which we know that the scaled fluctuations converge in law to the Tracy--Widom distribution. We will take the weights that appear in the log-gamma polymer.

Let $\theta>0$. The log-gamma polymer is the directed polymer with weights $1/X_{i,j}$, where the $X_{i,j}$ are i.i.d with the $\text{Gamma}(\theta,1)$ distribution, that is they have density
$$f(x)=\frac{1}{\Gamma(\theta)} x^{\theta-1} e^{-x} \quad (x>0).$$
The partition function is defined in the same way as before
$$Z_{n,\theta}=\sum_{\pi} \prod_{i=0}^{2n} \frac{1}{X_{i,\pi(i)}},$$
and the free energy is $\log Z_{n,\theta}$. Here we think of $\theta$ as a positive parameter which plays a similar role to the inverse temperature $\b$. In our case, we will take $\theta=1/\b^2$, with $\b=n^{-\alpha}$ for some $1/5<\alpha<1/4$.

This model was introduced by Sepp\"al\"ainen in \cite{seppa12}. He showed that with special boundary conditions and fixed $\theta$, the scaled free energy $(\log Z_{n,\theta})/2n$ converges almost surely to $-\Psi(\theta/2)$, and the fluctuations converge in distribution to the Tracy--Widom GUE distribution:
\begin{equation}\label{log-gamma limit}
\frac{\log Z_{n,\theta}+2n\Psi(\theta/2)}{-(\Psi''(\theta/2))^{1/3}n^{1/3}} \xrightarrow{d} TW_{GUE}.
\end{equation}
Here $\Psi(x)=\Gamma'(x)/\Gamma(x)$ is the digamma function. The special conditions on the boundary were removed by Borodin, Corwin and Remenik in \cite{bor-cor-rem13}, who also showed that \eqref{log-gamma limit} holds when $\theta$ depends on $n$ and converges to a finite limit (or in terms of $\b$, when $\b$ depends on $n$ and converges to a non-zero limit). Then in \cite[Theorem 2.1]{krish-qua18}, Krishnan and Quastel extended this to the intermediate disorder regime, where now $\theta \sim c/\b^2$ for some positive constant $c$, with $\b = n^{-\alpha}$ for some $\alpha<1/4$.

The digamma function satisfies the following asymptotics:
$$\Psi(x)=\log x-\frac{1}{2x}-\frac{1}{12x^2}+O\left(\frac{1}{x^4}\right), \qquad -\Psi''(x) \sim \frac{1}{x^2}$$
as $x \to \infty$ (see \cite[Chapter 6]{formulas92}). So \eqref{log-gamma limit} can be rewritten as
\begin{equation}\label{log-gamma limit 2}
\theta^{2/3}\left(\frac{\log Z_{n,\theta}+2n(\log \theta-\log 2-\frac{1}{\theta}-\frac{1}{3\theta^2})}{4^{1/3}n^{1/3}}\right) \xrightarrow{d} TW_{GUE}.
\end{equation}
The error term $O(1/\theta^4)$ for $\Psi(\theta/2)$ can be thrown away in our case since we are considering $1/5<\alpha<1/4$ and $\theta \sim c/\b^2=c n^{2\alpha}$. 

We now show that the weights in the standard polymer and the log-gamma polymer (renormalized appropriately) are both valid.

\begin{theo}
The following two parametrizations of weights are valid.
\begin{longlist}
\item[1.] $\omij(\b)=\psi_{i,j}(\b)^{-1}e^{\b \xi_{i,j}}$, where the $\xi_{i,j}$ are independent, have a uniform exponential tail and $\psi_{i,j}$ is the moment generating function of $\xi_{i,j}$.
\item[2.] $\omij(\b)=(\theta-1)/X_{i,j}$, where $\theta=\theta(\b) \sim \sigma^2/\b^2$ as $\b \to 0$ and $X_{i,j}$ are i.i.d with $X_{i,j} \sim \text{Gamma}(\theta,1)$.
\end{longlist}
\end{theo}

\begin{proof}
1. Properties 1 and 2 trivially hold. Here ``uniform exponential tail'' means there are constants $C,c>0$ such that for all $i,j$,
$$\P(|\xi_{i,j}|\ge \lambda) \le Ce^{-c\lambda}.$$
Therefore, for all $i,j$,
\begin{eqnarray*}
\psi_{i,j}(\b)&=&\E(e^{\b \xi_{i,j}}) \le \E(e^{\b|\xi_{i,j}|})=\int_0^{\infty} \P(e^{|\b \xi_{i,j}|}>t)dt\\
&=&1+\int_1^{\infty} \P\left(|\xi_{i,j}|>\frac{\log t}{\b}\right)dt \le 1+\int_1^{\infty} Ce^{-c\frac{\log t}{\b}} dt\\
&=&1+\frac{C\b}{c-\b},
\end{eqnarray*}
and
\begin{eqnarray*}
\psi_{i,j}(\b)&\ge& \E(e^{-\b|\xi_{i,j}|})=\int_0^{\infty} \P(e^{-|\b \xi_{i,j}|}>t)dt\\
&=&1-\int_0^1 \P\left(|\xi_{i,j}|\ge-\frac{\log t}{\b}\right)dt \ge 1-\int_0^1 Ce^{c\frac{\log t}{\b}} dt\\
&=&1-\frac{C\b}{c+\b}.
\end{eqnarray*}
These two inequalities imply that there is a positive constant $C$ such that for all $i,j$ and all sufficiently small $\b$,
\begin{equation}\label{psi inequality}
|\psi_{i,j}(\b)-1| \le C\b.
\end{equation}
In particular, for sufficiently small $\b$, we have $\psi_{i,j}(\b) \ge 1/2$ for all $i,j$. Therefore, by the triangle inequality for $L^k$ norms,
\begin{equation}\label{property 3 standard polymer}
\begin{split}
\E(|\om_{i,j}-1|^k)&=\frac{\E|e^{\b \xi_{i,j}}-\psi_{i,j}(\b)|^k}{\psi_{i,j}(\b)^k}\le 2^k \E|e^{\b \xi_{i,j}}-\psi_{i,j}(\b)|^k\\
&\le 2^k \left( (\E|e^{\b \xi_{i,j}}-1|^k)^{1/k}+(\E|\psi_{i,j}(\b)-1|^k)^{1/k} \right)^k\\
&\le 2^k \left( (\E|e^{\b \xi_{i,j}}-1|^k)^{1/k}+C\b \right)^k.
\end{split}
\end{equation}
By the mean value theorem, $|e^{\b \xi_{i,j}}-1| \le \b|\xi_{i,j}| e^{\b|\xi_{i,j}|}$. Since 
$$|\xi_{i,j}| \le \frac{2k}{c} e^{\frac{c|\xi_{i,j}|}{2k}},$$
it follows that for $\b<c/2k$, we have
$$
\E|e^{\b \xi_{i,j}}-1|^k\le \b^k \E(|\xi_{i,j}|^k e^{\b k|\xi_{i,j}|})\le \frac{2^k k^k}{c^k}\b^k \E(e^{c|\xi_{i,j}|})\le C_k \b^k
$$
where $C_k$ depends on $k$ but not $i$ or $j$. Plugging this back in \eqref{property 3 standard polymer} then gives Property 3.

After taking logs in \eqref{psi inequality}, we find
$$\log(1-C\b) \le \log \psi_{i,j} \le \log(1+C\b)$$
for all $i,j$. Since $\log(1-c\b) \sim -c\b$ and $\log(1+c\b) \sim c\b$ as $\b \to 0$, it follows that there is a constant $C$ such that for sufficiently small $\b$ and all $i,j$,
$$-C\b \le \log \psi_{i,j}(\b) \le C\b.$$
Therefore, if $0<s<1$, then
$$
\P(\om_{i,j} \ge e^{\b^s})=\P\left(\xi_{i,j} \ge \b^{s-1}+\frac{\log \psi_{i,j}(\b)}{\b}\right)\le \P(\xi_{i,j} \ge \b^{s-1}-C)
\le Ce^{-c\b^{s-1}}.
$$
A similar argument works to show that $\P(\om_{i,j} \le e^{-\b^s}) \le Ce^{-c \b^{s-1}}$, and these two inequalities imply Property 4.

2. Property 1 holds provided we only consider $\b$ small enough. Using well-known formulas for the moments of the Gamma distribution, we have $\E(X^{-1})=(\theta-1)^{-1}$, so Property 2 holds as well.

Next, by the Cauchy--Schwarz inequality,
$$\E\left|\frac{\theta-1}{X}-1\right|^k=\E\left|\frac{\theta-1-X}{X}\right|^k \le (\E(\theta-1-X)^{2k})^{1/2}\left(\E\left(\frac{1}{X^{2k}}\right)\right)^{1/2}.$$

The second expectation is
$$\left(\E\left(\frac{1}{X^{2k}}\right)\right)^{1/2}=\left(\frac{1}{(\theta-1) \cdots (\theta-2k)}\right)^{1/2}\le \frac{C_k}{\theta^k}.$$

Let $X_t$ be a Gamma process, i.e $X_t$ is a L\'evy process with independent increments given by $X_t-X_s \sim \text{Gamma}(t-s,1)$. Then $t-X_t$ is a square-integrable martingale, with quadratic variation $[X]_t=t$, so by the Burkholder--Davis--Gundy inequality (\cite[Theorem 4.4.20]{app04}),
$$(\E(\theta-1-X)^{2k})^{1/2}\le ((\E(\theta-X)^{2k})^{1/k}+1)^{k/2} \le C_k \E([X]_{\theta}^k)^{1/2}=C_k \theta^{k/2}.$$

Combining these 2 inequalities, we obtain Property 3:
$$\E\left|\frac{\theta-1}{X}-1\right|^k \le C_k \theta^{-k/2} \le C_k \b^k.$$

Finally, to verify Property 4, we use Chernoff's inequality. First, we have, for $\lambda>1$ and $r>0$,
\begin{eqnarray*}
\P\left(\frac{\theta-1}{X} \le \lambda^{-1} \right)&=&\P(X \ge (\theta-1)\lambda)\le \E(e^{rX}) e^{-r(\theta-1)\lambda}\\
&=&(1-r)^{-\theta}e^{-r(\theta-1)\lambda}=e^{-\theta \log(1-r)-r(\theta-1)\lambda}.
\end{eqnarray*}
Minimizing the last expression over $r$, we find $r=1-\theta/(\lambda (\theta-1))$, and substituting this back above yields
$$\P\left(\frac{\theta-1}{X} \le \lambda^{-1} \right) \le \exp \left( \theta(1-\lambda+\log \lambda)-\theta \log \frac{\theta}{\theta-1}+\lambda\right)$$
Now we take $\lambda=e^{\b^s}$ for some $0<s<1$ and we find that for all sufficiently small $\b>0$,
$$\P\left(\frac{\theta-1}{X} \le e^{-\b^s} \right)\le Ce^{-C\b^{2(s-1)}}.$$

Second, for $\lambda>1$ and $r>0$,
\begin{eqnarray*}
\P\left(\frac{\theta-1}{X} \ge \lambda \right)&=&\P(-X \ge -(\theta-1)\lambda^{-1})\le \E(e^{-rX}) e^{r(\theta-1)\lambda^{-1}}\\
&=&(1+r)^{-\theta}e^{r(\theta-1)\lambda^{-1}}=e^{-\theta \log(1+r)+r(\theta-1)\lambda^{-1}}.
\end{eqnarray*}
The minimum this time is achieved at $r=-1+\theta \lambda/(\theta-1)$, and substituting above gives
$$\P\left(\frac{\theta-1}{X} \ge \lambda \right) \le \exp \left( \theta(1-\lambda^{-1}+\log \lambda^{-1})-\theta \log \frac{\theta}{\theta-1}+\lambda^{-1}\right).$$
Taking $\lambda=e^{\b^s}$ for some $0<s<1$ again, we get that for every sufficiently small $\b>0$,
$$\P\left(\frac{\theta-1}{X} \ge e^{\b^s} \right) \le Ce^{-C\b^{2(s-1)}}.$$
So

$$\P\left(e^{-\b^s} \le \frac{\theta-1}{X} \le e^{\b^s}\right) \ge 1-Ce^{-C\b^{2(s-1)}},$$
and this is easily seen to imply Property 4.
\end{proof}
All that remains now to prove Theorem \ref{TW thm} is to show how to adjust the parameters so as to make the directed polymer match first and second moments with the log-gamma polymer, and then we can apply Theorem \ref{perturbation thm}. We can rewrite the free energy of the log-gamma polymer as 
$$\log \left( \sum_{\pi} \prod_{i=0}^{2n} \frac{1}{X_{i,\pi(i)}} \right)=\log \left( \sum_{\pi} \prod_{i=0}^{2n} \frac{\theta-1}{X_{i,\pi(i)}} \right)-2n\log(\theta-1).$$
Now $\log(\theta-1)$ has the asymptotics
$$\log(\theta-1)=\log\theta-\frac{1}{\theta}-\frac{1}{2\theta^2}+O\left(\frac{1}{\theta^3}\right)$$
as $\theta \to \infty$, so plugging all this back in \eqref{log-gamma limit 2}, we find
\begin{equation}\label{log-gamma limit 3}
\theta^{2/3}\left(\frac{F_{n,\theta}+2n(-\log 2+\frac{1}{3\theta^2})}{4^{1/3}n^{1/3}}\right) \xrightarrow{d} TW_{GUE},
\end{equation}
where $F_{n,\theta}$ is the normalized free energy of the log-gamma polymer
$$F_{n,\theta}=\log \left( \sum_{\pi} \prod_{i=0}^{2n} \frac{\theta-1}{X_{i,\pi(i)}} \right).$$
As before, the $O(1/\theta^3)$ error term can be thrown out when $\alpha>1/5$.

The second moment of the set of weights $\omij(\b)=(\theta-1)/X_{i,j}$ is
$$\E\left(\frac{\theta-1}{X_{i,j}}\right)^2=\frac{\theta-1}{\theta-2},$$
and for the weights $\omij'(\b)=e^{\b \xi_{i,j}}/\psi(\b)$,
$$\E\left(\frac{e^{\b \xi_{i,j}}}{\psi(\b)}\right)^2=\frac{\psi(2\b)}{\psi(\b)^2}.$$
Solving
$$\frac{\theta-1}{\theta-2}=\frac{\psi(2\b)}{\psi(\b)^2}$$
yields
$$\theta=2+\frac{\psi(\b)^2}{\psi(2\b)-\psi(\b)^2},$$
and so with this choice of $\theta$, the two sets of weights have the same first and second moments. Note that by expanding $\psi(\b)$ and $\psi(2\b)$ as Taylor series, we can see after some calculations that
\begin{equation}\label{theta expansion}
\theta=2+\frac{1+2\E(\xi)\b+O(\b^2)}{\text{Var}(\xi)\b^2+O(\b^3)},
\end{equation}
and so $\theta \sim 1/\sigma^2 \b^2$ as $\b \to 0$, where $\sigma^2$ is the variance of the $\xi_{i,j}$'s. We are thus in the context of Krishnan and Quastel's theorem, and therefore \eqref{log-gamma limit}, \eqref{log-gamma limit 2} and \eqref{log-gamma limit 3} all hold. By Theorem \ref{perturbation thm}, the free energies for the $\omij(\b)$ and $\omij'(\b)$ weights then both satisfy \eqref{log-gamma limit 3}. By \eqref{theta expansion},
$$\frac{1}{\theta^2}=\left(\frac{\text{Var}(\xi)\b^2+O(\b^3)}{1+2\E(\xi)\b+O(\b^2)}\right)^2=\text{Var}(\xi)^2 \b^4+O(\b^5).$$
Substituting this and $\theta \sim 1/\sigma^2 \b^2$ in \eqref{log-gamma limit 3} and factoring out the moment generating functions from the normalized free energy gives
$$\frac{\log Z_{n,\b}-2n(\log \psi(\b)+\log 2-\frac{\sigma^4 \b^4}{3})}{(4\sigma^4\b^4n)^{1/3}} \xrightarrow{d} TW_{GUE}$$
(once again we can throw out the $O(\b^5)$ error term using the fact that $\b=n^{-\alpha}$ for some $\alpha>1/5$). This concludes the proof of Theorem \ref{TW thm}!

\bibliographystyle{imsart-number} % Style BST file (imsart-number.bst or imsart-nameyear.bst)
\bibliography{bibliography_polymers.bib}       % Bibliography file (usually '*.bib')

%% or include bibliography directly:
% \begin{thebibliography}{}
% \bibitem{b1}
% \end{thebibliography}

\end{document}